\documentclass[12pt,a4paper,oneside]{article}
\usepackage[utf8]{inputenc}
\usepackage{amsmath}
\usepackage{amsfonts}
\usepackage{amssymb}
\usepackage[symbol]{footmisc}
\usepackage{stmaryrd}
\usepackage{graphicx}
\usepackage[english]{babel}
\usepackage{color}
\usepackage{tikz}
\usepackage{csquotes}
\usepackage{mathrsfs}
\usepackage{fancyhdr}
\usepackage[T1]{fontenc}
\usepackage{nicefrac}
\usepackage{dsfont}
\usepackage{amsthm}
\usepackage{pdfsync}
\usepackage[colorlinks=true,linkcolor=blue]{hyperref}
\makeatletter
\renewcommand*{\eqref}[1]{\hyperref[{#1}]{\textup{\tagform@{\ref*{#1}}}}}
\usetikzlibrary{calc,backgrounds,fit}
\usetikzlibrary{external}
\usetikzlibrary{patterns}
\usetikzlibrary{decorations.pathreplacing,decorations.markings,arrows}
\usetikzlibrary{shapes,intersections,positioning}
\usetikzlibrary{arrows}

\newcommand\centre[4][below]{\node (#3) at #2 [circle,minimum size=0.5em,inner sep=0pt,thin,fill,solid] {}; \node [#1=0.1em] at (#3) {#4};}
\newcounter{cst}
\newcommand{\ctel}[1]{K_{\refstepcounter{cst}\label{#1}\thecst}}
\newcommand{\cter}[1]{K_{\ref*{#1}}} 
\newcommand{\re}{\mathbb{R}}
\newcommand{\R}{\mathbb{R}}

\newcommand{\E}{\mathbb{E}}

\newcommand{\pe}{\psi_{\epsilon}}
\newcommand{\phe}{\phi_{\epsilon}}

\newcommand{\ups}{u_{\epsilon}}
\newcommand{\fnk}{f^n_K}
\newcommand{\fkk}{f^k_K}
\newcommand{\dt}{\Delta t}

\newcommand{\erwb}{\mathbb{E}\left[}
\newcommand{\erwe}{\right]}
\newcommand{\halbe}{\frac{1}{2}}

\newcommand{\Tau}{\mathcal{T}}
\newcommand{\edges}{\mathcal{E}}
\newcommand{\dkl}{d_{K|L}}
\newcommand{\edgesint}{\mathcal{E}_{\operatorname{int}}}
\newcommand{\edgesext}{\mathcal{E}_{\operatorname{ext}}}
\newcommand{\uhnl}{u_{h,N}^l}
\newcommand{\uhnr}{u_{h,N}^r}

\newcommand{\uhnrm}{u_{h_m,N_m}^r}

\newcommand{\lzlambda}{{L^2(\Lambda)}}

\newcommand{\whnl}{w_{h,N}^l}
\newcommand{\whnr}{w_{h,N}^r}

\newcommand{\erww}[1]{\mathbb{E}\left[{#1}\right]}
\newcommand{\reg}{\operatorname{reg}(\Tau)}

\newcommand{\tn}{t_n}
\newcommand{\tnp}{t_{n+1}}

\newcommand{\unk}{u^n_K}
\newcommand{\ukk}{u^k_K}
\newcommand{\unpk}{u^{n+1}_K}

\newcommand{\ukpk}{u^{k+1}_K}
\newcommand{\ukpl}{u^{k+1}_L}
\newcommand{\unpl}{u^{n+1}_L}

\newcommand{\di}{\displaystyle}

\newtheorem{defi}{Definition}[section]

\newtheorem{lem}[defi]{Lemma}
\newtheorem{teo}[defi]{Theorem}
\newtheorem{prop}[defi]{Proposition}

\theoremstyle{remark}
\newtheorem{remark}[defi]{Remark}

\numberwithin{equation}{section}
\title{Theoretical analysis of a finite-volume scheme for a stochastic Allen--Cahn problem with constraint}
\author{Caroline Bauzet\footnotemark[1], \and C\'edric Sultan\footnotemark[1], \and Guy Vallet \footnotemark[3] \and Aleksandra Zimmermann\footnotemark[4]}
\date{\today}
\begin{document}
\maketitle

\begin{abstract} 
The aim of this contribution is to address the convergence study of a time and space approximation scheme for an Allen-Cahn problem with constraint and perturbed by a multiplicative noise of It\^o type. The problem is set in a bounded domain of $\R^d$ (with $d=2$ or $3$) and homogeneous Neumann boundary conditions are considered. The employed strategy consists in building a numerical scheme on a regularized version \enquote{\`a la Moreau-Yosida} of the constrained problem, and passing to the limit simultaneously with respect to the regularization parameter and the time and space steps, denoted respectively by $\epsilon$, $\dt$ and $h$. Combining a semi-implicit Euler-Maruyama time discretization with a Two-Point Flux Approximation (TPFA) scheme for the spatial variable, one is able to prove, under the assumption $\dt=\mathcal{O}(\epsilon^{2+\theta})$ for a positive $\theta$, the convergence of such a \enquote{$(\epsilon, \dt, h)$} scheme towards the unique weak solution of the initial problem, \textit{ a priori} strongly in $L^2(\Omega;L^2(0,T;L^2(\Lambda)))$ and \textit{a posteriori} also strongly in $L^{p}(0,T; L^2(\Omega\times \Lambda))$ for any finite $p\geq 1$.\\


\end{abstract}

\footnotetext[1]{Aix Marseille Univ, CNRS, Centrale Med, LMA, Marseille, France, caroline.bauzet@univ-amu.fr, cedric.sultan@univ-amu.fr}
\footnotetext[3]{Univ Pau $\&$ Pays Adour, LMAP, UMR CNRS 5142, IPRA, Pau, France, guy.vallet@univ-pau.fr}
\footnotetext[4]{TU Clausthal, Institut f\"ur Mathematik, Clausthal-Zellerfeld, Germany,
aleksandra.zimmermann@tu-clausthal.de}

\section{Introduction}

We consider $\Lambda$ a bounded, open, connected, and polygonal set of $\R^d$ (with $d=2$ or $3$), and $(\Omega,\mathcal{A},\mathds{P})$ a probability space endowed with a right-continuous, complete filtration $(\mathcal{F}_t)_{t\geq 0}$.
For $T>0$, we are interested in finding a pair $(u,\psi)$ with $\psi\in \partial I_{[0,1]}(u)$, satisfying the following time noise-driven Allen-Cahn equation:
\begin{align}\label{equation}
\begin{aligned}
du+(\psi-\Delta u)\,dt &=g(u)\,dW(t)+(\beta(u)+f)\,dt, &&\text{ in }\Omega\times(0,T)\times\Lambda;\\
u(0,\cdot)&=u_0, &&\text{ in } \Omega\times\Lambda;\\
\nabla u\cdot \mathbf{n}&=0, &&\text{ on }\Omega\times(0,T)\times\partial\Lambda;
\end{aligned}
\end{align}
where $(W(t))_{t\geq 0}$ is a standard, one-dimensional Brownian motion with respect to $(\mathcal{F}_t)_{t\geq 0}$ on $(\Omega,\mathcal{A},\mathds{P})$ and $\mathbf{n}$ denotes the unit normal vector to $\partial\Lambda$ outward to $\Lambda$. As mentioned in \cite{BBBLV}\footnote{We remark that since only standard properties of $H^1(\Lambda)$ are required in \cite{BBBLV}, a Lipschitz regularity for the domain $\Lambda$ is enough.}, the sub-differential $\partial I_{[0,1]}$ represents a physical constraint on the solution of \eqref{equation} forcing it to remain bounded in $(0,1)$ by the presence of the Lagrange multiplier $\psi\in \partial I_{[0,1]}(u)$. More precisely, our equation in (\ref{equation}), can be written as a differential inclusion in the following manner:
\[\beta(u)+f-\partial_t\left(u-\int_0^\cdot g(u)\,dW\right)+\Delta u\in \partial I_{[0,1]}(u)\]
where the stochastic integral is understood in the sense of It\^o and the sub-differential of the indicator function $I_{[0,1]}:\R\rightarrow\R\cup 
\{+\infty\}$ defined by
\begin{align*}
I_{[0,1]}(r)=\begin{cases} 0 & \text{if} \ r\in [0,1]\\
+\infty & \text{else}
\end{cases}
\end{align*}
is the set-valued mapping $\partial I_{[0,1]}:[0,1]\rightarrow\mathcal{P}(\R)$ defined by 
\begin{align*}
\partial I_{[0,1]}(r)=\begin{cases} \{0\} & \text{if} \ r\in (0,1)\\
(-\infty,0] & \text{if} \ r=0\\
[0,\infty)& \text{if} \ r=1.
\end{cases}
\end{align*}
We consider the following assumptions on the data:
\begin{itemize}
\item[$\mathscr{A}_1$:] $u_0\in L^2(\Omega;L^2(\Lambda))$ is $\mathcal{F}_0$-measurable and verifies $0\leq u_0(\omega,x) \leq 1$, for almost all $(\omega,x)\in \Omega\times \Lambda$.

\item[$\mathscr{A}_2$:] $g:\re\rightarrow\re$ is a $L_g$-Lipschitz-continuous function (with $L_g\geq 0$), such that \\ $\operatorname{supp} g\subset [0,1]$.

\item[$\mathscr{A}_3$:] $\beta : \R\rightarrow \R$ is a $L_\beta$-Lipschitz-continuous function (with $L_\beta \geq 0$) such that for convenience $\beta(0)=0$.

\item[$\mathscr{A}_4$:] $f\in L^2_{\mathcal{P}_T}\big(\Omega\times(0,T);L^2(\Lambda)\big)\footnotemark[2]$.

\end{itemize}
\footnotetext[2]{For a given separable Banach space $X$, we denote by $L^2_{\mathcal{P}_T}\big(\Omega\times(0,T);X\big)$ the space of the predictable $X$-valued processes (\cite{DPZ14} p.94 or \cite{PrevotRockner} p.27). This space is the space $L^2\big(\Omega\times(0,T);X\big)$ for the product measure $d\mathds{P}\otimes dt$ on the predictable $\sigma$-field $\mathcal{P}_T$ (\textit{i.e.} the $\sigma$-field generated by the sets $\mathcal{F}_0\times \{0\}$ and the rectangles $A\times (s,t]$, for any $s,t\in[0,T]$ with $s\leq t$ and $A\in \mathcal{F}_s$).}

\begin{remark}
For technical reasons, we need the assumption $g(0)=g(1)=0$ as stated in $\mathscr{A}_2$. Without this assumption, we introduce additive noise to the problem setting. Then, already from a theoretical point of view, the analysis is much more complicated as one can see in, \textit{e.g.,} \cite{DM} and \cite{STVZ}.
\end{remark}

\subsection{Literature review}

The equation in \eqref{equation} is known in the literature as an Allen-Cahn type equation with constraint. It is applicable in modeling several physical phenomena, like phase transitions. In \cite{BBBLV}, a global existence and uniqueness result for this equation has been proposed to model the evolution of damage in continuum media. More precisely, it has been assumed that the solution $u$ to \eqref{equation} is a damage parameter, \textit{i.e.} the local proportion of active cohesive bonds in the micro-structure of a material. Then, the function $f$ on the right-hand side of \eqref{equation} denotes an external source of damage (mechanical or chemical), while the nonlinear source term $\beta$ is associated with the material's internal cohesion. A constraint was incorporated within the equation to restrict the values of $u$ to the interval $[0,1]$. This constraint has a physical meaning in the way that $u=1$ signifies that the material is completely undamaged, $u=0$ signifies that it is completely damaged while values of $u$ in $(0,1)$ represent varying degrees of intermediate damage. In \cite{BBBLV}, the physical constraint was ensured by the presence of a sub-differential graph, \textit{i.e.} a multivalued maximal monotone operator. In addition, a stochastic force term given by an It\^{o} integral has been added on the right-hand side of \eqref{equation}. Since its diffusion coefficient $g$ depends on the damage parameter, the stochastic force is said to be multiplicative. From a physical point of view, the presence of this random force term reflects the fact that the phenomenon of damage is related to microscopic changes in the structure and configuration of the material lattice as a consequence of breaking bonds and the formation of cavities and voids. These changes are clearly related to stochastic processes occurring at a microscopic level (as introduced in Ising materials), which we aimed to take into account in the macroscopic description.\\

 \noindent The literature on the deterministic Allen-Cahn equation is very rich, also including the presence of non-smooth (monotone) operators (see, among others, \cite{AC,CMZ,CGPS,DPKPJ,W}), and the stochastic case has been addressed by a growing number of surveys. Some recent results are devoted to questions of existence and uniqueness of solutions for stochastic Allen-Cahn equation \cite{RNT}, in \cite{BOS22} well posedness for stochastic Allen-Cahn type equations with $p$-Laplacian as well as the random separation property are studied. Others are more interested in questions of existence and regularity of solution for stochastic Cahn-Hilliard/Allen-Cahn problems \cite{AKM} and \cite{DPGS23}. The study of degenerate Kolmogorov equations and questions of ergodicity for stochastic Allen-Cahn equations with logarithmic potential have been considered in \cite{B21,SZ23}. Other authors have studied a stochastic Allen-Cahn-Navier-Stokes system with inertial effects and multiplicative noise of jump type in a bounded domain \cite{DPGS22,M23}. Let us mention that, in these last contributions, the sub-differential operator is replaced by smooth nonlinearities (possibly with prescribed growth conditions), as double-well potentials.\\

Furthermore, the study of stochastic partial differential inclusions in a rather general situation was carried out in \cite{BarbuRascanu,BensoussanRascanu,Rascanu}, as well as questions concerning transition semigroup and invariant measures in \cite{BarbuDaPrato}, or even obstacle problems with Lewy-{S}tampacchia's inequality for a stochastic {$T$}-monotone obstacle problem (see \cite{TV20}). According to \cite{BNSZ22}, the convergence analysis of numerical schemes for stochastic PDEs of parabolic type has been a very fashionable subject in recent decades and for this reason, an extensive literature on this topic is available (see \cite{ACQS20}, \cite{DP09} and \cite{OPW20} for a general overview). In the past the use of finite-element methods was the preferred technique for the spatial discretization of parabolic evolution equations (see \cite{Prohl}, \cite{BHL21} for a state of the art on this subject), and this is particularly true for the numerical approximation of stochastic Allen-Cahn type equations without constraint (in the chronological order let us mention the contributions \cite{FP03,QW19,BG20,ABNP21,LQ21,BGJK23,BP24}).\\

The numerical analysis of differential inclusions was first carried out on multivalued differential equations in 
\cite{V89,DL92,A94,M95}, and new studies on the subject have continued to be published ever since \cite{LV98,BS02, G03,BR07,BR10,R14,MT16}. In parallel, stochastic differential inclusions were studied in the early 2000's from a numerical point of view. Firstly, results of convergence analysis of time-discretization schemes have been derived in \cite{P00,B03,LN04, WZ13}, and secondly, convergence rate as long as error estimates have been investigated respectively in \cite{Z18} and \cite{EKKL21}. More recently, the time-space discretization of deterministic elliptic and parabolic partial differential inclusions was performed by combining Euler scheme with finite-element methods in \cite{R11,BR13,BER18}. The most challenging task in the numerical analysis of (stochastic or deterministic) partial differential inclusions is the proper formulation of the multivalued term. Often, the idea was not to determine approximate solutions of the inclusion itself but of approximate, single-valued equations (see, \textit{e.g.,} \cite{Sch86}, \cite{NP06}, \cite{LN04}). In addition to the spatial and temporal discretization parameter, an additional parameter for the regularisation of the multivalued term then appears. To obtain convergence of the scheme, one needs assumptions on the coupling of all discretization parameters.\\
 
To the best of our knowledge, the numerical analysis of stochastic partial differential inclusions is still an open topic. Our aim is then to fill the gap in the literature by addressing the first convergence analysis of a time and space discretization scheme for our stochastic Allen-Cahn problem with constraint (\ref{equation}). Let us precise that the main originality of our approach consists in the use of a finite-volume method for the spatial discretization instead of a finite-element one.

\subsection{Concept of solution and main result}
Following our previous work \cite{BBBLV}, we are interested here in the following concept of solution for Problem~\eqref{equation}:
\begin{defi}\label{solution} Any pair of stochastic processes $(u,\psi)\in \left(L^2_{\mathcal{P}_T}\big(\Omega\times(0,T);L^2(\Lambda)\big)\right)^2$ with $u$ belonging additionally to 
\begin{align*}
 L^2(\Omega;\mathscr{C}([0,T];L^2(\Lambda)))\cap L^2_{\mathcal{P}_T}\big(\Omega\times(0,T);H^1(\Lambda)\big),
\end{align*}
is a solution to Problem \eqref{equation} if almost everywhere in $(0,T)\times\Lambda$ and $\mathbb{P}$-a.s in $\Omega$, 
\[0\leq u \leq 1\text{ and }\psi\in \partial I_{[0,1]}(u),\]
and if the pair $(u,\psi)$ satisfies
\begin{align*}
u(t)=u_0+\int_0^t \big(\Delta u(s)-\psi(s)+\beta(u(s))+f(s)\big)\,ds+\int_0^t g(u(s))\,dW(s),
\end{align*}
in $L^2(\Lambda)$ and $\mathds{P}$-a.s in $\Omega$, where $\Delta$ denotes the Laplace operator on $H^1(\Lambda)$ associated with the formal Neumann boundary condition.
\end{defi}
\begin{remark}\label{240527_01}
\textit{A priori}, we have the predictability of $u$ with values in $L^2(\Lambda)$. It is a direct consequence of, \textit{e.g.,} \cite[Corollary 1.1.8]{HNVW16} that we may \textit{a posteriori} conclude that $u$ belongs to $ L^2_{\mathcal{P}_T}\big(\Omega\times(0,T);H^1(\Lambda)\big)$.
\end{remark}
Existence and uniqueness of a pair $(u,\psi)$ solution of Problem~\eqref{equation} in the sense of Definition \ref{solution} has been proved in \cite{BBBLV} for a given initial condition $u_0$ in $H^1(\Lambda)$ and under Assumptions $\mathscr{A}_2$, $\mathscr{A}_3$ and $\mathscr{A}_4$. To do so, we used a regularization procedure on the maximal monotone operator $\partial I_{[0,1]}$ by considering the following family of approximating problems depending on a parameter $\epsilon>0$:
\begin{align}\label{eqeps}
\begin{aligned}
d\ups+(\pe(\ups)-\Delta \ups)\,dt &=g(\ups)\,dW(t)+(\beta(\ups)+f)\,dt, &&\text{ in }\Omega\times(0,T)\times\Lambda;\\
\ups(0,\cdot)&=u_0, &&\text{ in } \Omega\times\Lambda;\\
\nabla \ups\cdot \mathbf{n}&=0, &&\text{ on }\Omega\times(0,T)\times\partial\Lambda;
\end{aligned}
\end{align}
where $\pe:\R\rightarrow \R$ denotes the Moreau-Yosida approximation of $\partial I_{[0,1]}$ (see, \textit{e.g.,} \cite{barbu76,brezis73}), defined for all $v\in\R$ by
\begin{eqnarray}\label{SPPR}
\pe(v)=-\frac{(v)^-}{\epsilon}+\frac{(v-1)^{+}}{\epsilon}=
\left\{\begin{array}{clll}
\di\frac{v}{\epsilon}&\text{if}&v\leq 0\\
0&\text{if}&v\in[0,1]\\
\di\frac{v-1}{\epsilon}&\text{if}&v\geq 1.
\end{array}\right.
\end{eqnarray}
Firstly, we proved in \cite{BBBLV}, for fixed $\epsilon>0$, existence and uniqueness of a solution $\ups$ for Problem (\ref{eqeps}) in the sense of Definition \ref{dups} below:
\begin{defi}\label{dups} A stochastic process $\ups\in \ L^2_{\mathcal{P}_T}\big(\Omega\times(0,T);H^1(\Lambda)\big)$ element of
\[
L^{\infty}\left(0,T;L^2(\Omega;H^1(\Lambda))\right) \cap L^2\left(\Omega;\mathscr{C}\left(0,T;L^2(\Lambda)\right)\right)\]
and such that $\displaystyle\partial_t\Big(\ups-\int_0^\cdot g(\ups)\,dW\Big)$ and $\Delta\ups$ belong to $L^2(\Omega;L^2( 0,T;L^2(\Lambda)))$, is a solution to the Problem (\ref{eqeps}) if almost everywhere in $(0,T)$ and $\mathbb{P}$-almost surely in $\Omega$, the following variational formulation holds for any $v\in H^1(\Lambda)$
\begin{align*}
\int_{\Lambda} \partial_t\left(\ups-\int_0^\cdot g(\ups) \,dW(s)\right)v \,dx+\int_\Lambda\nabla \ups.\nabla v \,dx+\int_\Lambda\pe(\ups) v \,dx=\int_\Lambda\big(\beta(\ups)+f\big)v \,dx,
\end{align*}
with $\displaystyle \mathbb{P}\text{-a.s in $\Omega$, }u_0=\lim_{t\rightarrow 0}u_\epsilon(\cdot,t)\text{ in }L^2(\Lambda).$
\end{defi}
Secondly, the analysis of the sequences $(\ups)_{\epsilon>0}$ and $(\pe(\ups))_{\epsilon>0}$ allowed us (mainly thanks to monotonicity tools) to pass to the limit with respect to the approximating parameter $\epsilon>0$. We finally proved existence of a solution $(u,\psi)$ of Problem (\ref{equation}) in the sense of Definition \ref{solution}, as a weak limit of a subsequence of the pair $(\ups,\pe(\ups))_{\epsilon>0}$. Then, we finished our study by proving the uniqueness of such a solution $(u,\psi)$.\\

\noindent The objective of the present paper is to propose a time and space approximation of the unique solution $(u,\psi)$ of Problem~\eqref{equation} in the sense of Definition \ref{solution}. To do so, our idea consists in discretizing, for a given $\epsilon>0$, Problem (\ref{eqeps}) by the way of finite-volume methods. Our main goal is to show that the resulting finite-volume approximation, depending on $\epsilon$ and on the time and space parameters (denoted respectively by $N$ and $h$ in the sequel), can be bounded independently of these three parameters, with the idea of making them tend towards zero simultaneously. The aim of the game is then to find a relationship between $\epsilon$ and $N$ that allows us to bound our finite-volume approximation and to pass to the limit in the numerical scheme, leading us to our main result stated hereafter. Thanks to the implicit time discretization of the equation's deterministic part, it seems natural that we may avoid an additional constraint linking $h$ to the other discretization parameters.

\begin{teo} \label{mainresult} 
Assume that hypotheses $\mathscr{A}_1$ to $\mathscr{A}_4$ hold and let $(u,\psi)$ be the unique solution of Problem \eqref{equation} in the sense of Definition~\ref{solution}. Let $(\Tau_m)_{m\in \mathbb{N}}$ be a sequence of admissible finite-volume meshes of $\Lambda$ in the sense of Definition \ref{defmesh} such that the mesh size $h_m$ tends to $0$, let $(N_m)_{m\in \mathbb{N}}\subset \mathbb{N}^{\star}$ be a sequence of positive integers which tends to infinity and let $(\epsilon_m)_{m\in\mathbb{N}}\subset \mathbb{R}_+^\star$ be another sequence such that $\lim_{m\rightarrow+\infty} \epsilon_m=0$. For a fixed $m\in\mathbb{N}$, let $u^r_{h_m,N_m}$ and $u^l_{h_m,N_m}$ be respectively the right and left in time finite-volume approximations defined by \eqref{eq:notation_wh}-\eqref{eq:def_u0}-\eqref{equationapprox} with $\Tau =\Tau_m$, $N=N_m$ and $\epsilon=\epsilon_m$. If there exists $\theta>0$ such that for any $m\in \mathbb{N}$, $\frac{T}{N_m}=\mathcal{O}((\epsilon_m)^{2+\theta})$, then the sequences $(\uhnrm)_{m\in \mathbb{N}}$ and $(\psi_{\epsilon_m}(\uhnrm))_{m\in \mathbb{N}}$ converge towards $u$ and $\psi$, respectively strongly and weakly in $L^2(\Omega;L^2(0,T;L^2(\Lambda)))$. Moreover, the convergence of $(\uhnrm))_{m\in \mathbb{N}}$ towards $u$ also holds strongly in $L^{p}(0,T; L^2(\Omega;L^2(\Lambda)))$ for any finite $p\geq 1$.
\end{teo}

\subsection{Outline}

This contribution is organized as follows. In Section \ref{sectiontwo}, the discretization framework is introduced: choice of the time step, definition of admissible finite-volume meshes of $\Lambda$, related notations and employed discrete norms. Then in Section \ref{sectiontwobis}, the semi-implicit TPFA scheme for the discretization of the regularized Problem (\ref{eqeps}) and its associated discrete solutions are defined, and the well-posedness of such a scheme is investigated. In Section \ref{estimates}, a clever relation between the regularization and the time and space discretization parameters (denoted $\epsilon, \dt$ and $h$, respectively), allows us to derive stability estimates satisfied by the discrete solutions, independently of these three parameters. Section \ref{ConvFVscheme} is then dedicated to the convergence analysis of our scheme by combining arguments we developed in \cite{BSZ23} for the passage to the limit with respect to $\dt$ and $h$, with the ones used in \cite{BBBLV} to pass to the limit with respect to $\epsilon$.

\section{Discretization framework}\label{sectiontwo}

Let us start this section by some general notations, then Subsections \ref{mesh}, \ref{discretenotation}, \ref{dnadg} contain all the definitions and notations related to temporal and spatial discretizations. Let us mention that they are the same as in our previous papers \cite{BNSZ22,BNSZ23,BSZ23}, but for a matter of self-containedness we choose to repeat them identically.

\subsection{General notations}

\begin{itemize}
\item[$\bullet$] The integral over $\Omega$ with respect to the probability measure $\mathds{P}$ is denoted by $\E[\cdot]$, and is called the expectation.
\item[$\bullet$] For any $x,y$ in $\R^d$, the euclidean norm of $x$ is denoted by $|x|$, and the associated scalar product of $x$ and $y$ by $x\cdot y$.
\item[$\bullet$] The $d$-dimensional Lebesgue measure of $\Lambda$ is denoted by $|\Lambda|$, by overusing the euclidean norm notation.
\item[$\bullet$] For $q\in\{1,d\}$, the $L^\infty(\R^q)$ norm is denoted by $||\cdot||_\infty$.
\end{itemize}

\subsection{Uniform time step and admissible finite-volume meshes}\label{mesh}
With the aim of proposing a time and space approximation of the variational solution of Problem (\ref{eqeps}), a choice for the temporal and spatial discretizations must be made. The temporal one is achieved using a uniform subdivision: setting $N\in \mathbb{N}^{\star}$, the fixed time step is defined by $\dt=\frac{T}{N}$ and the interval $[0,T]$ is decomposed in $0=t_0<t_1<...<t_N=T$ equidistantly with $\tn=n \dt$ for all $n\in \{0, ..., N-1\}$. For the spatial one, following \cite[Definition 9.1]{EymardGallouetHerbinBook}, we consider admissible finite-volume meshes as defined hereafter: 
\begin{defi} (Admissible finite-volume mesh)\label{defmesh} 
An admissible finite-volume mesh of $\Lambda$, denoted by $\Tau$, is given by a family of \enquote{control volumes}, which are open polygonal convex subsets of $\Lambda$, a family of subsets of $\overline{\Lambda}$ contained in hyperplanes of $\R^d$, denoted by $\mathcal{E}$ (these are the edges for $d=2$ or sides for $d=3$ of the control volumes), with strictly positive $(d-1)$-dimensional Lebesgue measure, and a family of points of $\Lambda$ denoted by $\mathcal{P}$ satisfying the following properties\footnotemark[1]\footnotetext[1]{In fact, we shall denote, somewhat incorrectly, by $\mathcal{T}$ the family of control volumes.}
\begin{itemize}
\item $\overline{\Lambda}=\bigcup_{K\in\Tau}\overline{K}$.
\item For any $K\in\mathcal{T}$, there exists a subset $\mathcal{E}_K$ of $\mathcal{E}$ such that $\partial K=\overline{K}\setminus K=\bigcup_{\sigma\in \mathcal{E}_K}\overline{\sigma}$ and $\mathcal{E}=\bigcup_{K\in\mathcal{T}}\mathcal{E}_K$. $\mathcal{E}_K$ is called the set of edges of $K$ for $d=2$ and sides for $d=3$, respectively.
\item For any $K,L\in\Tau$, with $K\neq L$ then either the $(d-1)$ Lebesgue measure of $\overline{K}\cap \overline{L}$ is $0$ or $\overline{K}\cap \overline{L}=\overline{\sigma}$ for some $\sigma\in\mathcal{E}$, which will then be denoted by $K|L$ or $L|K$.
\item The family $\mathcal{P}=(x_K)_{K\in\mathcal{T}}$ is such that $x_K\in \overline{K}$ for all $K\in\mathcal{T}$ and, if $K,L\in\mathcal{T}$ are two neighbouring control volumes, it is assumed that $x_K\neq x_L$, and that the straight line between $x_K$ and $x_L$ is orthogonal to $\sigma=K|L$.
\item For any $\sigma\in\mathcal{E}$ such that $\sigma\subset\partial\Lambda$, let $K$ be the control volume such that $\sigma \in\mathcal{E}_K$. If $x_K\notin\sigma$, the straight line going through $x_K$ and orthogonal to $\sigma$ has a nonempty intersection with $\sigma$.
\end{itemize}
\end{defi}
\pagebreak
\begin{figure}[htbp!]
\centering
\begin{tikzpicture}[scale=2]

  \clip (-1.2,-0.6) rectangle (1.8,1.3);

  \node[rectangle,fill] (A) at (-1,0.6) {};
  \node[rectangle,fill] (B) at (0,1.2) {};
  \node[rectangle,fill] (C) at (0,-0.2) {};
  \node[rectangle,fill] (D) at (1.5,0.3) {};

  \centre[above right]{(-0.6,0.5)}{xK}{$x_K$};
  \centre[above left]{(0.9,0.5)}{xL}{$x_L$};
  
  \draw[thick] (B)--(C) node [pos=0.7,right] {$\sigma=$\small{$K|L$}};

  \draw[thin,opacity=0.5] (A) -- (B) -- (C) -- (A) ;
  \draw[thin,opacity=0.5] (D) -- (B) -- (C) -- (D);

  \draw[dashed] (xK) -- (xL);
  
  \coordinate (KK) at ($(xK)!0.65!-90:(xL)$);
  \coordinate (LL) at ($(xL)!0.65!90:(xK)$);

  \draw[dotted,thin] (xK) -- (KK);
  \draw[dotted,thin] (xL) -- (LL);

  \draw[|<->|] (KK) -- (LL) node [midway,fill=white,sloped] {$\dkl$};
 
  \coordinate (KAB) at ($(A)!(xK)!(B)$);
  \coordinate (KAC) at ($(A)!(xK)!(C)$);

  \coordinate (LDB) at ($(D)!(xL)!(B)$);
  \coordinate (LDC) at ($(D)!(xL)!(C)$);

  \draw[dashed] (xK) -- ($(xK)!3!(KAB)$);
  \draw[dashed] (xK) -- ($(xK)!3!(KAC)$);

  \draw[dashed] (xL) -- ($(xL)!3!(LDB)$);
  \draw[dashed] (xL) -- ($(xL)!3!(LDC)$);

  \begin{scope}[on background layer]   
    \draw (0,0.5) rectangle ++ (0.1,-0.1);
  \end{scope}
  
  \coordinate (nkl) at ($(B)!0.35!(C)$);
  \draw[->,>=latex] (nkl) -- ($(nkl)!0.3cm!90:(C)$) node[above] {$\mathbf{n}_{K,\sigma}$};
\end{tikzpicture}
\caption{Notations of the mesh $\mathcal T$ associated with $\Lambda\subset\mathbb{R}^2$\label{fig:notation_mesh}}
\end{figure}
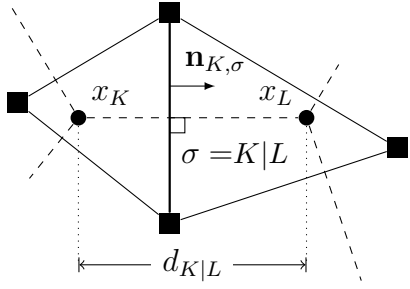
For a given admissible finite-volume mesh $\Tau$ of $\Lambda$, the following associated notations will be used in the rest of the paper.\\

\textbf{Notations.}
\begin{itemize}

\item The mesh size is denoted by $h=\operatorname{size}(\Tau)=\sup\{\operatorname{diam}(K): K\in\Tau\}$.

\item The number of control volumes $K\in\Tau$ is denoted by $d_h\in\mathbb{N}$, where $h=\operatorname{size}(\Tau)$.

\item The sets of interior and exterior interfaces are respectively denoted by\\
$\mathcal{E}_{\operatorname{int}}:=\{\sigma\in\mathcal{E}:\sigma\nsubseteq \partial\Lambda\}$ and $\mathcal{E}_{\operatorname{ext}}:=\{\sigma\in\mathcal{E}:\sigma\subseteq \partial\Lambda\}$.

\item For any $K\in\Tau$, the $d$-dimensional Lebesgue measure of $K$ is denoted by $m_K$.

\item For any $K\in\Tau$, the unit normal vector to $\partial K$ outward to $K$ is denoted by $\mathbf{n}_K$.

\item For any $K\in\Tau$ and any $\sigma \in \mathcal{E}_K$, the unit vector on $\sigma$ pointing out of $K$ is denoted by $\mathbf{n}_{K,\sigma}$.

\item For any $\sigma\in\edgesint$, the $(d-1)$-dimensional Lebesgue measure of $\sigma$ is denoted by $m_\sigma$.

\item For any neighboring control volumes $K,L\in\Tau$, the euclidean distance between $x_K$ and $x_L$ is denoted by $d_{K|L}$.

\item The maximum of edges incident to any vertex of the mesh is denoted by $\mathcal N$.

\item For any $K\in \Tau$ and any $\sigma\in \mathcal{E}_K$, the Euclidean distance between $x_K$ and $\sigma$ is denoted by $d(x_K,\sigma)$.
\end{itemize}

The regularity of the mesh $\Tau$ is measured by the following positive number 
\begin{align*}\label{mrp}
\reg=\max\left(\mathcal N,\max_{\scriptscriptstyle K \in\Tau \atop \sigma\in\mathcal{E}_K} \frac{\operatorname{diam}(K)}{d(x_K,\sigma)}\right).
\end{align*}
As in the deterministic setting, it is assumed that $\reg$ is uniformly bounded by a constant not depending on the mesh size $h$, which is one of the key point to prove the convergence of our finite-volume scheme. Indeed, the introduction of the number $\reg$ allows us particularly to derive the following inequality:

\begin{equation}\label{hoverdkl}
\forall K,L\in \Tau, \frac{h}{\dkl}\leq \reg,
\end{equation}
and such a uniform control of the ratio in the left-hand side of (\ref{hoverdkl}) will be essential in the proof of Lemma \ref{keylemma}.

\subsection{Discrete unknowns and piecewise constant functions }\label{discretenotation}
From here to the end of Section \ref{estimates}, let $N$ be a positive integer, $\dt=\frac{T}{N}$ and $\Tau$ be an admissible finite-volume mesh of $\Lambda$ in the sense of Definition \ref{defmesh} with a mesh size $h>0$. The ideology of a finite-volume method to approximate the variational solution of Problem \eqref{eqeps} is to associate to each control volume $K\in\Tau$ and time $t_n\in \{0, ..., t_N\}$ a discrete unknown value denoted by $\unk\in \mathbb{R}$, expected to be an approximation of $u_\epsilon(t_n,x_K)$.\\
 
For a given vector $(w_K^n)_{K\in\Tau}\in\re^{d_h}$, we introduce in what follows various associated functions. Firstly, we define the piecewise constant function in space $w_h^n:\Lambda\rightarrow \mathbb{R}$ by
\[
w_h^n(x):=\sum_{K\in\Tau} w^n_K \mathds{1}_K(x),\ \forall x\in \Lambda.
\]
Using the fact that the mesh $\Tau$ is fixed, the continuous mapping defined from $ \mathbb{R}^{d_h}$ to $L^2(\Lambda)$ by
\[(w^n_K)_{{K\in\Tau}} 
\mapsto \sum_{K\in\Tau}\mathds{1}_K w^n_K, \]
allows us to consider the space $\re^{d_h}$ as a finite-dimensional subspace of $L^2(\Lambda)$ and to do the following natural identification between the function and the vector
\[
w^n_h\equiv(w^n_K)_{K\in\Tau}\in \mathbb{R}^{d_h}.
\]
Secondly, the knowledge for any $n \in\{0,\ldots,N\}$ of the function $w_h^n$ enables us to define the following right and left piecewise constant functions in time and space denoted respectively by $\whnr$ and $\whnl$
by
\begin{equation}\label{eq:notation_wh}
\begin{aligned}
\whnr(t,x):=&\sum_{n=0}^{N-1} w_h^{n+1}(x)\mathds{1}_{[t_n,t_{n+1})}(t)\text{ if }t\in[0,T)
\text{ and } \whnr(T,x):=w_h^N(x),\\
\whnl(t,x):=&\sum_{n=0}^{N-1} w_h^n(x)\mathds{1}_{[t_n,t_{n+1})}(t) \text{ if }t\in(0,T]\text{ and }
\whnl(0,x):=w_h^0(x).
\end{aligned}
\end{equation}

Since $\Tau$ and $N$ are fixed, reasoning as for the piecewise constant function in space above, the continuity of the mapping defined from $\mathbb{R}^{d_h\times N}$ to $L^2(0,T;L^2(\Lambda))$ by
\[(w_K^n)_{\substack{K\in\Tau \\ n\in\{0,\ldots,N-1\}}}\mapsto\sum_{\substack{K\in\Tau \\ n\in\{0,\ldots,N-1\}}}\mathds{1}_K\mathds{1}_{[t_n,t_{n+1})}w_K^n,\]
allows us to consider the space $\mathbb{R}^{d_h\times N}$ as a finite-dimensional subspace of $L^2(0,T;L^2(\Lambda))$ and to do naturally the identifications 
\begin{align*}
\whnl&\equiv(w_K^n)_{\substack{K\in\Tau \\ n\in\{0,\ldots,N-1\}}}\in \mathbb{R}^{d_h\times N},\\
\whnr&\equiv(w_K^{n+1})_{\substack{K\in\Tau \\ n\in\{0,\ldots,N-1\}}}\in \mathbb{R}^{d_h\times N}.
\end{align*}
\begin{remark}
In the following, when a time and space function $\phi:[0,T]\times \Lambda\rightarrow \mathbb{R}$ will be considered for fixed $x\in \Lambda$, the space variable will be omitted in the notations and $\phi(x)$ will be written instead of $\phi(\cdot,x)$. An analogous notation will apply for fixed $t\in [0,T]$, \textit{i.e.}, we will write $\phi(t)$ for $\phi(t,\cdot)$. 
\end{remark}

\subsection{Discrete norms and weak gradient}\label{dnadg}

For the remainder of this subsection, let us set $n\in\{0, ..., N-1\}$, consider an arbitrary vector $(w^n_K)_{K\in\Tau}\in \mathbb{R}^{d_h}$ and identify it with the piecewise constant function in space $w^n_h\equiv(w^n_K)_{K\in\Tau}$. In the following, we introduce the definitions of the discrete $L^2(\Lambda)$-norm, the weak gradient and the discrete $H^1(\Lambda)$-semi-norm for $w^n_h$.

\begin{defi}[Discrete $L^2(\Lambda)$-norm]The discrete $L^2(\Lambda)$-norm of $w^n_h \in\re^{d_h}$ is defined by
\[
||w^n_h||_{L^2(\Lambda)}=\left(\sum_{K\in \Tau}m_K |w^n_K|^2\right)^\frac12.
\]
\end{defi}

\begin{defi}[Weak gradient]
Let $e_h$ be the number of elements in $\mathcal{E}$. The weak gradient operator $\nabla^h : \re^{d_h}\rightarrow (\re^d)^{e_h}$ maps any scalar fields $w^n_h\in\re^{d_h}$ into vector fields $\nabla^h w^n_h=(\nabla_\sigma^h w^n_h)_{\sigma\in\edges}\in (\re^d)^{e_h}$, where for any $\sigma\in\edges$, $\nabla_\sigma^h w^n_h\in \R^d$ is defined by
\[
\nabla_\sigma^h w^n_h :=
\left\{
\begin{aligned}
d\frac{w^n_L-w^n_K}{\dkl} \mathbf{n}_{K,\sigma}, \quad
&\text{ if }\sigma=K|L\in\edgesint ; \\
\qquad 0, \qquad\qquad &\text{ if } \sigma\in\edgesext.
\end{aligned}
\right.
\]
\end{defi}

\begin{defi}[Discrete $H^1(\Lambda)$-semi-norm]
The discrete $H^1(\Lambda)$-semi-norm of $w^n_h \in\re^{d_h}$ is defined by
 \[
|w^n_h|_{1,h}:=\left(\sum_{\sigma=K|L\in\edgesint}\frac{m_\sigma}{\dkl}|w^n_K-w^n_L|^2\right)^\halbe.
 \]
\end{defi}

\begin{remark}\label{remarkforuhnrboundiii} 
We have the following relation between the discrete $(L^2(\Lambda))^d$-norm of $\nabla^h w^n_h$ and the discrete $H^1(\Lambda)$-semi-norm of $w^n_h$:
\begin{align}\label{linkgradsnh1}
\|\nabla^h w^n_{h}\|_{(L^2(\Lambda))^d}^2
=\sum_{\sigma=K|L\in\edgesint}\frac{d_{K|L}m_\sigma}{d}\left|d\frac{w^{n}_K-w^{n}_L}{d_{K|L}}\right|^2=d|w^n_{h}|_{1,h}^2.
\end{align}
\end{remark}
We end this subsection by recalling a classical trick of sum reordering, which will be used several times in the rest of the paper.
\begin{remark}[Discrete partial integration]\label{discrpartint}
For any $\widetilde w^n_h\equiv(\widetilde{w}^n_K)_{K\in\Tau}\in \mathbb{R}^{d_h}$, the following rule of "discrete partial integration" holds:
\begin{align}\label{PInt}
\sum_{K\in\Tau}\sum_{\sigma=K|L\in\edges_K\cap\edgesint}\frac{m_\sigma}{\dkl}(w^n_K-w^n_L)\widetilde w^n_K
=\sum_{\sigma=K|L\in\edgesint}\frac{m_\sigma}{\dkl}(w^n_K-w^n_L)(\widetilde w^n_K-\widetilde w^n_L).
\end{align}
\end{remark}

Now, we have all the necessary definitions and notations to present the finite-volume scheme studied in this paper. This is the aim of the next section.

\section{The semi-implicit TPFA scheme for Problem (\ref{eqeps})}\label{sectiontwobis}
By the discretization of the initial condition $u_0$ of Problem \eqref{equation} over each control volume:
\begin{align}
\label{eq:def_u0}
u_K^0:=\frac{1}{m_K}\int_K u_0(x)\,dx, \quad \forall K\in \Tau,
\end{align}
we are firstly allowed to define the random vector $u_h^0\equiv (u^0_K)_{K\in\Tau} \in \re^{d_h}$.
Secondly, starting from this given initial $\mathcal{F}_0$-measurable random vector $u_h^0\in\re^{d_h}$, and fixing a parameter $\epsilon>0$, we construct our semi-implicit TPFA scheme as follows:\\
\quad\\
For any $n \in \{0,\dots,N-1\}$, knowing $u_h^n\equiv (u^{n}_K)_{K\in\Tau} \in \re^{d_h}$, we search for $u_h^{n+1}\equiv(u^{n+1}_K)_{K\in\Tau}\in\re^{d_h}$, 
solution of the following equations, $\mathbb{P}$-a.s in $\Omega$:
\begin{align}\label{equationapprox}
\begin{split}
&\frac{m_K}{\dt}(u_K^{n+1}-\unk) +\sum_{\sigma=K|L\in\edgesint\cap\edges_K}\frac{m_\sigma}{\dkl}(u_K^{n+1}-\unpl)+m_K\pe(\unpk)\\
&=\frac{m_K}{\dt}g(\unk)(W^{n+1}-W^n)+m_K\beta(\unpk)+m_K\fnk,\quad \forall K\in \Tau,
\end{split}
\end{align}
where $W^{n+1}-W^n$ denotes the increments of the Brownian motion between $t_{n+1}$ and $t_n$:
\[
W^{n+1}-W^n=W(t_{n+1})-W(t_n)\text{ for }n\in\{0,\dots,N-1\},
\]
and $\fnk$ is defined by
\begin{eqnarray}\label{deffnk}
\fnk=\frac{1}{\dt\, m_K}\int_{\tn}^{\tnp}\int_K f(x,t)\,dx\,dt.
\end{eqnarray}
\begin{remark} Although for any $n\in \{1, ..., N\}$ and any $K\in\Tau$, the discrete unknowns $u^n_K$ (and then the discrete solution $u_h^n$) depend on $\epsilon$, we omit this dependency from the notation for the sake of clarity.
\end{remark}
\begin{prop}[Well-posedness of the scheme]
\label{210609_prop1}
 Let $\Tau$ be an admissible finite-volume mesh of $\Lambda$ in the sense of Definition \ref{defmesh} with a mesh size $h$, let $N$ be a positive integer and let $\epsilon\in\R_+^\star$ be a given parameter. Then, under Assumptions $\mathscr{A}_1$ to $\mathscr{A}_4$, there exists a unique solution $(u_h^n)_{1\le n \le N} \in (\re^{d_h})^N$ to Problem~\eqref{equationapprox} associated with the initial vector $u^0_h$ defined by~\eqref{eq:def_u0}. Moreover, for any $n\in \{0,\ldots,N\}$, $u_h^n$ is a $\mathcal{F}_{t_n}$-measurable random vector.
\end{prop}

\begin{proof} It is an adaptation of the main result of \cite{BNSZ23} in the particular case where the convection term is equal to zero and the source term is equal to $u\mapsto\beta(u)-\pe(u)$.
\end{proof}
The right and left finite-volume approximations $\uhnr$ and $\uhnl$ defined by \eqref{eq:notation_wh} to approximate the solution $u_{\epsilon}$ of Problem (\ref{eqeps}) are then built from the discrete solution $(u_h^n)_{1\le n \le N} \in (\re^{d_h})^N$ given by Proposition \ref{210609_prop1}.

\section{Stability estimates}\label{estimates}

In this section, several stability estimates, satisfied by the discrete solution $(u_h^n)_{1 \le n \le N} \in (\re^{d_h})^N$ given by Proposition \ref{210609_prop1}, and also by the associated left and right finite-volume approximations $(\uhnl)_{h,N}$ and $(\uhnr)_{h,N}$ defined by \eqref{eq:notation_wh}, will be derived. Let us start by bounding the discrete initial data: 
\begin{lem}
\label{bound_u0}Under Assumption $\mathscr{A}_1$, the discrete initial data $u_h^0 \in \re^{d_h}$ associated to $u_0$ and defined by~\eqref{eq:def_u0} satisfies the following inequality:
\begin{equation*}
\erwb\|u_h^0\|_{L^2(\Lambda)}^2\erwe \leq \erwb\|u_0\|^2_{L^2(\Lambda)}\erwe.
\end{equation*}
\end{lem}
\begin{proof} It is a direct consequence of the definition of $u_h^0$ and Cauchy-Schwarz inequality.
\end{proof}
This first lemma allows us to obtain the following first bounds on the discrete solutions:

\begin{prop}[Bounds on the discrete solutions]\label{bounds}
There exists a constant $K_0>0$, depending only on $u_0$, $L_g$, $L_{\beta}$, $f$ and $T$ such that for any $\epsilon>0$, any $N\in\mathbb{N}^{\star}$ large enough (depending on $L_{\beta}$) and any $h\in \R^\star_+$
\begin{align*}
&\erwb \|u_h^n\|_{L^2(\Lambda)}^2 \erwe+\sum_{k=0}^{n-1}\erwb\|u_h^{k+1}-u_h^k\|_{L^2(\Lambda)}^2\erwe+\dt \sum_{k=0}^{n-1}\erww{|u_h^{k+1}|_{1,h}^2}\leq K_0,\; \forall n\in \{1,\ldots,N\}.
\end{align*}
\end{prop}

\begin{proof}
Set $\epsilon>0$, $N\in\mathbb{N}^{\star}$, $h\in \mathbb{R}_+^\star$ and fix $n\in \{1,\ldots,N\}$. For any $k\in \{0,\ldots,n-1\}$, we multiply the numerical scheme \eqref{equationapprox} with $\ukpk$, take the expectation, and sum over $K\in\Tau$ to obtain thanks to \eqref{PInt}
\begin{align}\label{implicitscheme}
\begin{split}
&\sum_{K\in\Tau}\frac{m_K}{\dt}\erwb(\ukpk-\ukk)\ukpk\erwe
+ \sum_{\sigma=K|L\in\edgesint}\frac{m_\sigma}{\dkl}\erwb|\ukpk-\ukpl|^2\erwe\\
&+\sum_{K\in\Tau}m_K\,\erwb \pe(\ukpk)\ukpk\erwe\\
=\,&\sum_{K\in\Tau}\frac{m_K}{\dt}\erwb g(\ukk)\ukpk\left(W^{k+1}-W^k\right)\erwe+\sum_{K\in\Tau}m_K\,\erwb \big(\beta(\ukpk)+\fkk\big)\ukpk\erwe.
\end{split}
\end{align}
We consider the terms of \eqref{implicitscheme} separately. Firstly note that
\begin{align}\label{term1}
\sum_{K\in\Tau}\frac{m_K}{\dt}\erwb(\ukpk-\ukk)\ukpk\erwe=\halbe\sum_{K\in\Tau}\frac{m_K}{\dt}\erwb|\ukpk|^2-|\ukk|^2+|\ukpk-\ukk|^2\erwe.
\end{align}
Secondly, since $\pe$ is monotone with $\pe(0)=0$, one gets that
\begin{eqnarray}\label{term2}
\sum_{K\in\Tau}m_K\,\erwb \pe(\ukpk)\ukpk\erwe&\geq& 0.
\end{eqnarray}
Thirdly, since $\ukk$ and $\left(W^{k+1}-W^k\right)$ are independent one obtains
\begin{eqnarray*}
\sum_{K\in\Tau}\frac{m_K}{\dt}\erwb g(\ukk)\ukk\left(W^{k+1}-W^k\right)\erwe=0,
\end{eqnarray*}
and so by applying Young's inequality and using the It\^{o} isometry one arrives at
\begin{align}\label{term3}
\begin{split}
&\sum_{K\in\Tau}\frac{m_K}{\dt}\erwb g(\ukk)\ukpk\left(W^{k+1}-W^k\right)\erwe\\
=\,&\sum_{K\in\Tau}\frac{m_K}{\dt}\erwb g(\ukk)(\ukpk-\ukk)\left(W^{k+1}-W^k\right)\erwe\\
\leq\,&\sum_{K\in\Tau}\frac{m_K}{\dt}\erwb|g(\ukk)\left(W^{k+1}-W^k\right)|^2\erwe+\frac{1}{4}\sum_{K\in\Tau}\frac{m_K}{\dt}\erwb|\ukpk-\ukk|^2\erwe\\
\leq\, &\dt L_g^2\sum_{K\in\Tau}\frac{m_K}{\dt}\erwb|\ukk|^2\erwe+\frac{1}{4}\sum_{K\in\Tau}\frac{m_K}{\dt}\erwb|\ukpk-\ukk|^2\erwe.
\end{split}
\end{align}
Fourthly, using the Lipschitz property of $\beta$ with $\beta(0)=0$, the following holds
\begin{eqnarray}\label{term4}
\sum_{K\in\Tau}m_K\,\erwb \beta(\ukpk)\ukpk\erwe&\leq& L_{\beta}\sum_{K\in\Tau}m_K\,\erwb |\ukpk|^2\erwe.
\end{eqnarray}
Fifthly, 
\begin{align}\label{term5}
\hspace*{-0.5cm}\sum_{K\in\Tau}m_K\,\erwb \fkk\ukpk\erwe
\leq& \frac{1}{2}\sum_{K\in\Tau}m_K\,\erwb |\fkk|^2\erwe+ \frac{1}{2}\sum_{K\in\Tau}m_K\,\erwb |\ukpk|^2\erwe\nonumber\\
\leq&\frac{1}{2\dt}\sum_{K\in\Tau} \erwb \int_{t_k}^{t_{k+1}}\int_K|f(x,t)|^2\,dx\,dt\erwe+ \frac{1}{2}\sum_{K\in\Tau}m_K\,\erwb |\ukpk|^2\erwe.
\end{align}
\noindent Combining \eqref{term1}-\eqref{term2}-\eqref{term3}-\eqref{term4} and \eqref{term5} and multiplying the obtained inequality with $2\dt$, one gets
\begin{eqnarray*}
&&\sum_{K\in\Tau}m_K\,\erwb |\ukpk|^2-|\ukk|^2+|\ukpk-\ukk|^2\erwe+ 2\dt\sum_{\sigma=K|L\in\edgesint}\frac{m_\sigma}{\dkl}\erwb|\ukpk-\ukpl|^2\erwe\\
&\leq&2\dt L_g^2\sum_{K\in\Tau}m_K\,\erwb|\ukk|^2\erwe+\frac{1}{2}\sum_{K\in\Tau}m_K\,\erwb|\ukpk-\ukk|^2\erwe\\
&&+\dt (2L_{\beta}+1)\sum_{K\in\Tau}m_K\,\erwb |\ukpk|^2\erwe+\sum_{K\in\Tau} \erwb \int_{t_k}^{t_{k+1}}\int_K|f(x,t)|^2\,dx\,dt\erwe.
\end{eqnarray*}
Then, 
\begin{align*}
&\big(1-\dt(2 L_{\beta}+1)\big)\sum_{K\in\Tau}m_K\,\erwb |\ukpk|^2-|\ukk|^2\erwe
+\frac{1}{2}\sum_{K\in\Tau}m_K\,\erwb|\ukpk-\ukk|^2\erwe\\
&+ 2\dt\sum_{\sigma=K|L\in\edgesint}\frac{m_\sigma}{\dkl}\erwb|\ukpk-\ukpl|^2\erwe\\
\leq& \dt 
(2{L_g}^2+2L_{\beta}+1)\sum_{K\in\Tau}m_K\,\erwb (\ukk)^2\erwe+\sum_{K\in\Tau} \erwb \int_{t_k}^{t_{k+1}}\int_K|f(x,t)|^2\,dx\,dt\erwe.
\end{align*}
For $\dt$ small enough so that $1-\dt(2 L_{\beta}+1)\geq \frac{1}{4}$, after summing over $k\in\{0,\dots,n-1\}$, one arrives at
\begin{align}\label{lem1beschr}
\begin{split}
&\frac{1}{4}\erwb\|u_h^n\|_{L^2(\Lambda)}^2-\|u_h^0\|_{L^2(\Lambda)}^2\erwe+\frac{1}{2}\sum_{k=0}^{n-1}\erww{\|u_h^{k+1}-u_h^k\|_{L^2(\Lambda)}^2}+ 2\dt\sum_{k=0}^{n-1}\erwb|u_h^{k+1}|_{1,h}^2\erwe\\
&\leq \dt (2{L_g}^2+2L_{\beta}+1) \sum_{k=0}^{n-1}\erwb\|u_h^k\|_{L^2(\Lambda)}^2\erwe+\sum_{k=0}^{n-1}\sum_{K\in\Tau} \erwb \int_{t_k}^{t_{k+1}}\int_K|f(x,t)|^2\,dx\,dt\erwe.
\end{split}
\end{align}
Then, it follows that
\begin{align*}
&\erwb\|u_h^n\|_{L^2(\Lambda)}^2\erwe\\
&\leq\erwb\|u_h^0\|_{L^2(\Lambda)}^2\erwe+4\dt(2L_g^2+2L_{\beta}+1)\sum_{k=0}^{n-1}\erwb\|u_h^k\|_{L^2(\Lambda)}^2\erwe+4||f||^2_{L^2\left(\Omega;L^2(0,T;L^2(\Lambda))\right)}.
\end{align*}
Applying the discrete Gronwall lemma yields 
\begin{align}\label{uhnbound}
\erwb\|u_h^n\|_{L^2(\Lambda)}^2\erwe\leq\left(\erwb\|u_h^0\|_{L^2(\Lambda)}^2\erwe+4||f||^2_{L^2\left(\Omega;L^2(0,T;L^2(\Lambda))\right)}\right)e^{4T(2L_g^2+2L_{\beta}+1)}.
\end{align}
From \eqref{uhnbound} and Lemma~\ref{bound_u0} we may conclude that there exists a constant $\Upsilon>0$ such that
\begin{align}\label{uhnboundbis}
\sup_{n\in\{1,\dots,N\}}\erww{\|u_h^n\|_{L^2(\Lambda)}^2}\leq \Upsilon.
\end{align}
Thanks to \eqref{uhnboundbis} one gets that for all $n\in\{1,\ldots N\}$
\begin{align}\label{210819_02}
\dt\sum_{k=0}^{n-1}\erww{\|g(u_h^k)\|_{L^2(\Lambda)}^2}\leq L_g^2\dt\sum_{k=0}^{n-1}\erww{\|u_h^k\|_{L^2(\Lambda)}^2}\leq L_g^2T\Upsilon.
\end{align}
From \eqref{lem1beschr}, Lemma~\ref{bound_u0} and \eqref{uhnboundbis} it now follows that for all $n\in \{1,\ldots, N\}$
\begin{align*}
\begin{split}
&\erwb\|u_h^n\|_{L^2(\Lambda)}^2\erwe+2\sum_{k=0}^{n-1}\erww{\|u_h^{k+1}-u_h^k\|_{L^2(\Lambda)}^2}+ 8\dt\sum_{k=0}^{n-1}\erwb|u_h^{k+1}|_{1,h}^2\erwe\\
&\leq \erwb\|u_0\|_{L^2(\Lambda)}^2\erwe+4\Upsilon T(2L_g^2+2L_{\beta}+1)+4||f||^2_{L^2\left(\Omega;L^2(0,T;L^2(\Lambda))\right)}.
\end{split}
\end{align*}
\end{proof}
We are now interested in the bounds on the right and left finite-volume approximations defined by~\eqref{eq:notation_wh}.

\begin{lem}\label{210611_lem01}
The sequences $(\uhnr)_{\epsilon,h,N}$ and $(\uhnl)_{\epsilon,h,N}$ are bounded in $L^\infty(0,T;L^2(\Omega;L^2(\Lambda)))$, independently of the regularization and discretization parameters $\epsilon\in \mathbb{R}_+^\star$, $N\in\mathbb{N}^{\star}$ and $h\in \mathbb{R}_+^\star$. Additionally, $(\uhnl)_{\epsilon,h,N}$ is bounded in $L^2_{\mathcal{P}_T}\big(\Omega\times(0,T);L^2(\Lambda)\big)$.
\end{lem}
\begin{proof} 
We note that by \eqref{uhnboundbis}
\begin{align*}
\Vert u_{h,N}^r\Vert_{L^{\infty}(0,T;L^2(\Omega;L^2(\Lambda)))}+\Vert u_{h,N}^l\Vert_{L^{\infty}(0,T;L^2(\Omega;L^2(\Lambda)))}&\leq 2\sup_{n\in \{0,1,\ldots,N\}}\erww{\Vert u_h^n\Vert^2_{L^2(\Lambda)}}\\
&\leq 2\Upsilon+\erww{\Vert u_0\Vert^2_{L^2(\Lambda)}}
\end{align*}
The predictability of $(\uhnl)_{\epsilon,h,N}$ with values in $\lzlambda$ is a consequence of the $\mathcal{F}_{\tn}$ -measurability of $\unk$ for all $n\in\{0, ..., N\}$ and all $K\in\Tau$. Indeed, by construction, $(\uhnl)_{\epsilon,h,N}$ is then an elementary process adapted to the filtration $(\mathcal{F}_t)_{t\geq 0}$ and so it is predictable.
\end{proof}
\begin{remark}
Note that by Proposition~\ref{bounds}, one gets the following useful estimate
\begin{align}\label{210824_05}
\erww{\|\uhnr-\uhnl\|_{L^2(0,T;\lzlambda)}^2}=\dt\,\erww{\sum_{n=0}^{N-1}\|u_h^{n+1}-u_h^n\|_\lzlambda^2}
\leq K_0\,\dt,
\end{align}
\end{remark}
Thanks to Proposition~\ref{bounds} we can also obtain a $L^2(\Omega;L^2(0,T;L^2(\Lambda)))$-bound on the weak gradients of the finite-volume approximation $(\uhnr)_{\epsilon,h,N}$:
\begin{lem}\label{remarkuhnrboundintomega}
There exists a constant $\ctel{K1}\geq 0$ depending only on $u_0$, $L_g$, $L_\beta$, $f$ and $T$
such that
\begin{align}\label{uhnrboundinotomega}
\int_0^T\erwb|\uhnr(t)|_{1,h}^2\erwe\,dt\leq \cter{K1}.
\end{align}
\end{lem}
\begin{proof} It is a direct consequence of Proposition \ref{bounds}.
\end{proof}

\begin{lem}\label{boundguhnlr}
The sequences $(g(\uhnr))_{\epsilon,h,N}$, $(g(\uhnl))_{\epsilon,h,N}$, $(\beta(\uhnr))_{\epsilon,h,N}$, and $(\beta(\uhnl))_{\epsilon,h,N}$ are bounded in $L^2(\Omega;L^2(0,T;\lzlambda))$ independently of the regularization and discretization parameters $\epsilon\in \mathbb{R}_+^\star$, $N\in\mathbb{N}^{\star}$ and $h\in \mathbb{R}_+^{\star}$. Moreover, $(g(\uhnl))_{\epsilon,h,N}$ and $(\beta(\uhnl))_{\epsilon,h,N}$ are predictable processes with values in $L^2(\Lambda)$.
\end{lem}
\begin{proof} It is a direct consequence of the boundedness of the sequences $(\uhnr)_{\epsilon,h,N}$ and $(\uhnl)_{\epsilon,h,N}$ in $L^2(\Omega;L^2(0,T;L^2(\Lambda)))$ given by Lemma \ref{210611_lem01} and of the Lipschitz nature of $g$ and $\beta$.
\end{proof}
\begin{lem}\label{boundsnguhnlr}
There exists a constant $\ctel{K3}\geq 0$ depending only on $u_0$, $L_g$, $L_\beta$, $f$ and $T$ such that 
\begin{align}\label{uhnrboundinotomegabis}
\int_0^T\erwb|g(\uhnr(t))|_{1,h}^2\erwe\,dt\leq \cter{K3}.
\end{align}
\end{lem}
\begin{proof}After noticing that:
\begin{eqnarray*}
\int_0^T\erwb|g(\uhnr(t))|_{1,h}^2\erwe\,dt& \leq& L_g^2\int_0^T\erwb|\uhnr(t)|_{1,h}^2\erwe\,dt,
\end{eqnarray*}
the result is immediate thanks to Lemma \ref{remarkuhnrboundintomega}.
\end{proof}
\begin{prop}\label{boundpe}
If we assume that there exists $\theta >0$ such that $\dt=\mathcal{O}(\epsilon^{2+\theta})$ and that $0<\epsilon<1$, then the sequences $(\pe(\uhnr))_{\epsilon,h,N}$ and $(\pe(\uhnl))_{\epsilon,h,N}$ are bounded respectively in $L^2(\Omega;L^2(0,T;\lzlambda))$ and $L^2_{\mathcal{P}_T}\big(\Omega\times(0,T);L^2(\Lambda)\big)$, independently of the regularization and discretization parameters $\epsilon\in \mathbb{R}_+^\star$, $N\in\mathbb{N}^{\star}$ and $h\in \mathbb{R}_+^\star$.
\end{prop}
\begin{proof} Setting $0<\epsilon<1$, $K\in \Tau$, $N\in \mathbb{N}_+^\star$ and $n\in \{0, ..., N-1\}$, multiplying \eqref{equationapprox} by $\dt\pe(\unpk)$, taking the expectation, summing over $K\in \Tau$ and over $n\in\{0,\dots,N-1\}$ lead to
\begin{align}\label{scmpeunpk}
\begin{split}
&\sum_{n=0}^{N-1} \sum_{K\in \Tau}m_K\,\erwb (u_K^{n+1}-\unk)\pe(\unpk)\erwe +\sum_{n=0}^{N-1} \sum_{K\in \Tau}m_K\,\dt\,\erwb(\pe(\unpk))^2\erwe\\
&+\sum_{n=0}^{N-1}\dt \sum_{K\in \Tau}\erwb\sum_{\sigma=K|L\in\edgesint\cap\edges_K}\frac{m_\sigma}{\dkl}(u_K^{n+1}-\unpl)\pe(\unpk)\erwe\\
=&\sum_{n=0}^{N-1} \sum_{K\in \Tau} m_K\,\erwb g(\unk)(W^{n+1}-W^n)\pe(\unpk)\erwe\\
&+\sum_{n=0}^{N-1} \sum_{K\in \Tau} \dt\, m_K\,\erwb \big(\beta(\unpk)+\fnk\big)\pe(\unpk)\erwe.
\end{split}
\end{align}
Let us study separately each term of \eqref{scmpeunpk}.\\
$\bullet$ For the study of the first term, we introduce the convex antiderivative of $\pe$ defined for any $v\in \R$ by
 \begin{eqnarray}\label{defphe}
\phe(v)=
\left\{\begin{array}{clll}
\di\frac{v^2}{2\epsilon}&\text{if}&v\leq 0\\
0&\text{if}&v\in[0,1]\\
\di\frac{(v-1)^2}{2\epsilon}&\text{if}&v\geq 1.
\end{array}\right.
\end{eqnarray}
Note that thanks to the convexity of $\phe$, the following holds
\[
(\unpk-\unk)\pe(\unpk)=(\unpk-\unk)\phe'(\unpk)\geq \phe(\unpk)-\phe(\unk),
\]
and so
\begin{align}\label{T1}
\begin{split}
&\sum_{n=0}^{N-1} \sum_{K\in \Tau}m_K\,\erwb (u_K^{n+1}-\unk)\pe(\unpk)\erwe\geq 0,\\
 \end{split}
\end{align}
owing to the facts that $\di\sum_{K\in \Tau}m_K\,\erwb \phe(u^{N}_K)\erwe \geq 0$ and $\di\sum_{K\in \Tau}m_K\,\erwb \phe(u^{0}_K)\erwe=0$ since from Assumption $\mathscr{A}_1$, $\mathbb{P}$-a.s in $\Omega$ and for any $K\in \Tau$, $0\leq u^0_K \leq 1$.\bigskip\\
$\bullet$ Using \eqref{PInt} and the monotonicity of $\pe$, one proves that
\begin{align}\label{T2}
\begin{split}
&\sum_{n=0}^{N-1}\dt \sum_{K\in \Tau}\erwb\sum_{\sigma=K|L\in\edgesint\cap\edges_K}\frac{m_\sigma}{\dkl}(u_K^{n+1}-\unpl)\pe(\unpk)\erwe\\
=&\sum_{n=0}^{N-1} \dt\,\erwb \sum_{\sigma=K|L\in\edgesint}\frac{m_\sigma}{\dkl}(u_K^{n+1}-\unpl)\left( \pe(\unpk)-\pe(\unpl)\right)\erwe
\\ \geq& 0.
\end{split}
 \end{align}
 $\bullet$ Firstly, by using the mean value theorem and the fact that for any $v\in \R$, $\psi_{\epsilon}(v)g(v)=0$, we can prove that there exist several elements between $\unk$ and $\unpk$, all written in the form $\zeta^{n+1}_K=(1-\lambda^{n+1}_K)\unpk+\lambda^{n+1}_K \unk$ (for some $\lambda^{n+1}_K\in [0,1]$), and such that
 the following inequality holds true
\begin{eqnarray}\label{MVT}
\Big|g(\unk)\big(\pe(\unpk)-\pe(\unk)\big)\Big| &\leq& \frac1\epsilon \big|g(\unk)-g(\zeta^{n+1}_K)\big||\unpk-\unk|.
\end{eqnarray} 
Indeed, this can be shown by separating the cases according to the position of $\unk$ and $\unpk$ relative to $0$ and $1$:
\begin{itemize}
\item[-] If $\unk \not\in (0,1)$, by setting $\zeta^{n+1}_K=\unk$ we get
\begin{align*}
0=\Big|g(\unk)\big(\pe(\unpk)-\pe(\unk)\big)\Big| =\frac1\epsilon \big|g(\unk)-g(\zeta^{n+1}_K)\big||\unpk-\unk|.
\end{align*}
\item[-] If $\unk \in (0,1)$ and $\unpk \in [0,1]$, by setting again $\zeta^{n+1}_K=\unk$, 
\begin{align*}
0=\Big|g(\unk)\big(\pe(\unpk)-\pe(\unk)\big)\Big| = \frac1\epsilon \big|g(\unk)-g(\zeta^{n+1}_K)\big||\unpk-\unk|.
\end{align*}
\item[-] If $\unk \in (0,1)$ and $\unpk<0$, $\pe(\unk)=\pe(0)$ and there exists $\zeta^{n+1}_K \in (\unpk,0)$ such that 
\begin{align*}
g(\unk)\big(\pe(\unpk)-\pe(\unk)\big) =& \ g(\unk)\big(\pe(\unpk)-\pe(0)\big) \\
=& \ g(\unk)\pe'(\zeta^{n+1}_K)\unpk \\
=& \ \big(g(\unk)-g(\zeta^{n+1}_K)\big)\pe'(\zeta^{n+1}_K)\unpk
\\ \text{and }
\Big|g(\unk)\big(\pe(\unpk)-\pe(\unk)\big)\Big| \leq& \ \frac1\epsilon \big|g(\unk)-g(\zeta^{n+1}_K)\big||\unpk-\unk|
\end{align*}
where we used that $g(\zeta^{n+1}_K)=0$ and $|\unpk| \leq |\unpk-\unk|$.
\item[-] If $\unk \in (0,1)$ and $\unpk>1$, $\pe(\unk)=\pe(1)$ and there exists $\zeta^{n+1}_K \in (1,\unpk)$ such that 
\begin{align*}
g(\unk)\big(\pe(\unpk)-\pe(\unk)\big) &= g(\unk)\pe'(\zeta^{n+1}_K)(\unpk-1)\\
&= \big(g(\unk)-g(\zeta^{n+1}_K)\big)\pe'(\zeta^{n+1}_K)(\unpk-1)
\\ \text{and }
\Big|g(\unk)\big(\pe(\unpk)-\pe(\unk)\big)\Big|& \leq \frac1\epsilon \big|g(\unk)-g(\zeta^{n+1}_K)\big||\unpk-\unk|
\end{align*}
by arguments similar to the previous case.
\end{itemize}
Using the assumption that there exists $\theta >0$ such that $\dt=\mathcal{O}(\epsilon^{2+\theta})$, 
then for any natural number $p$ such that $p \geq 1+\dfrac2\theta$, one has that (since $0<\epsilon <1$)
\begin{eqnarray}\label{borneratiop}
&&\dfrac{\dt^{p-1}}{\epsilon^{2p}}=\dfrac{\dt^{p-1}}{\epsilon^{(2+\theta)(p-1)}}\epsilon^{\theta(p-1)-2 }\leq \Big(\dfrac{\dt}{\epsilon^{2+\theta}}\Big)^{p-1}.
\end{eqnarray}
Choosing $p\in \mathbb{N}$ according to \eqref{borneratiop}, using successively the fact that $g(\unk)\pe(\unk)=0$, Inequality (\ref{MVT}), Young's inequality (with $p$ and its conjugate $\frac{p}{p-1}$), the constant $K_0$ given by Proposition \ref{bounds}, the fact that
\[
\E\left[\big(W^{n+1}-W^n\big)^{2p}\right]=\frac{(2p)!}{p!2^p}\dt^p=(2p-1)!! \dt^p
\]
and Inequality (\ref{borneratiop}), one gets the existence of a constant $C_p>0$ only depending on $p$ such that
\begin{align}\label{T3}
&\sum_{n=0}^{N-1} \sum_{K\in \Tau} m_K\,\erwb g(\unk)(W^{n+1}-W^n)\pe(\unpk)\erwe\nonumber
\\=&
\sum_{n=0}^{N-1} \sum_{K\in \Tau} m_K\,\erwb g(\unk)(W^{n+1}-W^n)(\pe(\unpk)-\pe(\unk))\erwe\nonumber
\\ \leq&
\frac{1}{2}\sum_{n=0}^{N-1} \sum_{K\in \Tau}m_K\,\erwb \left(\frac{W^{n+1}-W^n}{\epsilon}\right)^2 \big(g(\unk)-g(\zeta^{n+1}_K)\big)^2\erwe+\frac{1}{2}\sum_{n=0}^{N-1} \sum_{K\in \Tau}m_K\,\erwb \big(\unpk-\unk\big)^2\erwe\nonumber
\\ \leq & 
\frac{1}{2p}\sum_{n=0}^{N-1} \sum_{K\in \Tau}m_K\,\erwb\left(\frac{W^{n+1}-W^n}{\epsilon}\right)^{2p}\erwe+
\frac{p-1}{2p}\sum_{n=0}^{N-1} \sum_{K\in \Tau}m_K\,\erwb\big(g(\unk)-g(\zeta^{n+1}_K)\big)^{\frac{2p}{p-1}}\erwe+K_0\nonumber
\\ = & 
\frac{(2p-1)!!}{2p \epsilon^{2p}}\sum_{n=0}^{N-1} \sum_{K\in \Tau}m_K\,\dt^p+
 \frac{p-1}{2p} L_g^2\big(2||g||_{\infty}\big)^{\frac{2}{p-1}}\sum_{n=0}^{N-1} \sum_{K\in \Tau}m_K\,\erwb\big(\unk-\zeta^{n+1}_K\big)^{2}\erwe+K_0\nonumber
 \\ \leq & 
C_p \frac{\dt^{p-1}}{\epsilon^{2p}}\sum_{n=0}^{N-1} \sum_{K\in \Tau}m_K\,\dt+
L_g^2\big(2||g||_{\infty}\big)^{\frac{2}{p-1}}\sum_{n=0}^{N-1} \sum_{K\in \Tau}m_K\,\erwb\big(\unk-\unpk)^2\erwe+K_0\nonumber
\\ \leq &
 C_p \Big(\dfrac{\dt}{\epsilon^{2+\theta}}\Big)^{p-1}|\Lambda| T+ K_0\Big(L_g^2\big(2||g||_{\infty}\big)^{\frac{2}{p-1}}+1\Big),
\end{align}
where $(2p-1)!!$ denotes the double factorial of $2p-1$, the product of all odd integers up to $2p-1$.\\
\centerline{}
$\bullet$ The second right-hand side term of \eqref{scmpeunpk} can be handled in the following manner thanks to Young's inequality and the constant $K_0$ given by Proposition \ref{bounds}:
\begin{align}\label{T4}
&\sum_{n=0}^{N-1} \sum_{K\in \Tau}\dt\, m_K\,\erwb \left(\beta(\unpk)+\fnk\right)\pe(\unpk)\erwe\nonumber\\
 \leq &\frac{1}{2}\sum_{n=0}^{N-1} \sum_{K\in \Tau}\dt\, m_K\,\erwb \left(\pe(\unpk)\right)^2\erwe+\frac{1}{2}\sum_{n=0}^{N-1} \sum_{K\in \Tau}\dt\, m_K\,\erwb \left(\beta(\unpk)+\fnk\right)^2\erwe\nonumber\\
\leq &\frac{1}{2}\sum_{n=0}^{N-1} \sum_{K\in \Tau}\dt\, m_K\,\erwb \left(\pe(\unpk)\right)^2\erwe + L_\beta^2 TK_0+ ||f||^2_{L^2(\Omega;L^2(0,T;\lzlambda))}.
\end{align}
Finally, combining (\ref{T1}),(\ref{T2}),(\ref{T3}) and (\ref{T4}), we obtain that 
\begin{align*}
&\frac{1}{2}\sum_{n=0}^{N-1} \sum_{K\in \Tau}\dt\, m_K\,\erwb \left(\pe(\unpk)\right)^2\erwe\\
\leq &
C_p \Big(\dfrac{\dt}{\epsilon^{2+\theta}}\Big)^{p-1}|\Lambda| T
+ K_0\Big(L_g^2\big(2||g||_{\infty}\big)^{\frac{2}{p-1}}+1\Big) + L_\beta^2 T K_0 + ||f||^2_{L^2(\Omega;L^2(0,T;\lzlambda))}
\end{align*}
and the announced result holds since $\dt=\mathcal{O}(\epsilon^{2+\theta})$.
\end{proof}
\begin{remark}\label{boundpeuhnl}
Note that using the constant $K_0>0$ given by Proposition \ref{bounds}, the following inequality holds directly
\begin{eqnarray}\label{controldiffpe(uhnl)pe(uhnr)}
||\pe(\uhnl)-\pe(\uhnr)||^2_{L^2(\Omega;L^2(0,T;\lzlambda))}&\leq &\frac{\dt}{\epsilon^2}K_0.
\end{eqnarray}
\end{remark}
\begin{lem}\label{buhnrppanp}
If we assume that there exists $\theta > 0$ such that $\dt=\mathcal{O}(\epsilon^{2+\theta})$ and that $0<\epsilon<1$, then the sequences $\left(\frac{(\uhnr)^-}{\epsilon}\right)_{\epsilon,h,N}$ and $\left(\frac{(\uhnr-1)^+}{\epsilon}\right)_{\epsilon,h,N}$ are bounded in \\$L^2(\Omega;L^2(0,T;\lzlambda))$, independently of the regularization and discretization parameters $\epsilon\in \mathbb{R}_+^\star$, $N\in\mathbb{N}^{\star}$ and $h\in \mathbb{R}_+^\star$.
\end{lem}
\begin{proof} Since $\di\pe(\uhnr)=-\frac{(\uhnr)^-}{\epsilon}+\frac{(\uhnr-1)^{+}}{\epsilon}$ and $(\uhnr)^-\times(\uhnr-1)^+=0$, one gets that
\[\left|\left|-\frac{(\uhnr)^-}{\epsilon} \right|\right|^2_{L^2(\Omega;L^2(0,T;\lzlambda))}+\left|\left|\frac{(\uhnr-1)^+}{\epsilon} \right|\right|^2_{L^2(\Omega;L^2(0,T;\lzlambda))}\\
=\big|\big|(\pe(\uhnr))\big|\big|^2_{L^2(\Omega;L^2(0,T;\lzlambda))},\]
and the result holds directly since the right-hand side is bounded by Proposition \ref{boundpe}.
\end{proof}
Using the same technique, one proves the following:
\begin{lem}\label{buhnlppanp}
If we assume that there exists $\theta >0$ such that $\dt=\mathcal{O}(\epsilon^{2+\theta})$ and that $0<\epsilon<1$, then the sequences $\left(\frac{(\uhnl)^-}{\epsilon}\right)_{\epsilon,h,N}$ and $\left(\frac{(\uhnl-1)^+}{\epsilon}\right)_{\epsilon,h,N}$ are bounded in \\$L^2_{\mathcal{P}_T}\big(\Omega\times(0,T);L^2(\Lambda)\big)$, independently of the regularization and discretization parameters $\epsilon\in \mathbb{R}_+^\star$, $N\in\mathbb{N}^{\star}$ and $h\in \mathbb{R}_+^\star$.
\end{lem}
\begin{remark}\label{remlimcomu}Note that using Lemma \ref{buhnrppanp}, Lemma \ref{buhnlppanp} and by expanding the square term of \eqref{controldiffpe(uhnl)pe(uhnr)}, one can prove that
\begin{eqnarray*}
\left|\left|\frac{(\uhnr)^-}{\epsilon}-\frac{(\uhnl)^-}{\epsilon} \right|\right|^2_{L^2(\Omega;L^2(0,T;\lzlambda))}&\leq &\frac{\dt}{\epsilon^2}K_0 \\
\text{ and }\left|\left|\frac{(\uhnr-1)^+}{\epsilon}-\frac{(\uhnl-1)^+}{\epsilon} \right|\right|^2_{L^2(\Omega;L^2(0,T;\lzlambda))}&\leq &\frac{\dt}{\epsilon^2}K_0,
\end{eqnarray*}
which assures us that if the sequences $\left(\frac{(\uhnr)^-}{\epsilon}\right)_{\epsilon}$ and $\left(\frac{(\uhnl)^-}{\epsilon}\right)_{\epsilon}$ (respectively $\left(\frac{(\uhnr-1)^+}{\epsilon}\right)_\epsilon$ and $\left(\frac{(\uhnr-1)^+}{\epsilon}\right)_\epsilon$) converge, it is necessarily towards a common limit.
\end{remark}

\section{Convergence of the \enquote{$(\epsilon, \dt, h)$} scheme}\label{ConvFVscheme}

We have now all the necessary tools to pass to the limit in our \enquote{$(\epsilon, \dt, h)$} scheme. In what follows, let $(\Tau_m)_{m\in\mathbb{N}}$ be a sequence of admissible meshes of $\Lambda$ in the sense of Definition \ref{defmesh} such that the mesh size $h_m$ tends to $0$ when $m$ tends to $+\infty$, let $(N_m)_{m\in\mathbb{N}} \subset\mathbb{N}^\star$ be a sequence with $\lim_{m\rightarrow+\infty} N_m=+\infty$, set $\dt_m:=\frac{T}{N_m}$, and let $(\epsilon_m)_{m\in\mathbb{N}}\subset ]0,1[$ be another sequence such that $\lim_{m\rightarrow+\infty} \epsilon_m=0$, and assume that there exists $\theta>0$ such that $\dt_m=\mathcal{O}(\epsilon^{2+\theta}_m)$ for any $m\in \mathbb{N}$.
\\ 
For the sake of simplicity, for $m\in\mathbb{N}$, we shall use the notations $\Tau=\Tau_m$, $h=\operatorname{size}(\Tau_m)$, $\dt=\dt_m$, $N=N_m$ and $\epsilon=\epsilon_m$ when the $m$-dependency is not useful for the understanding of the reader.

\subsection{Weak convergences of finite-volume approximations}
First of all, owing to the bounds on the discrete solutions obtained in the previous section, we are able to derive the following weak convergences:

\begin{prop}\label{addreg u} 
There exists a process $u\in L^2_{\mathcal{P}_T}\big(\Omega\times(0,T);H^1(\Lambda)\big)$ such that, up to subsequences of $(\uhnr)_m$ and $(\uhnl)_m$ denoted in the same way,
\[\uhnl\rightarrow u \ \text{and} \ \uhnr\rightarrow u,\text{ both weakly in }L^2(\Omega;L^2(0,T;L^2(\Lambda)))\text{ as }m\rightarrow+\infty.\]
\end{prop}
\begin{proof} We refer to \cite[Proposition 4.1]{BSZ23}, since the proof is exactly the same.
\end{proof}

\begin{lem}\label{CVpeuhnr}
There exists a process $\psi$ in $L^2_{\mathcal{P}_T}\big(\Omega\times(0,T);L^2(\Lambda)\big)$ such that, up to subsequences of $(\pe(\uhnr))_m$ and $(\pe(\uhnl))_m$ denoted in the same way,
\[
\pe(\uhnr)\rightarrow \psi \text{ and }\pe(\uhnl)\rightarrow \psi,\text{ both weakly in }L^2(\Omega;L^2(0,T;L^2(\Lambda)))\text{ as }m\rightarrow+\infty.
\]
\end{lem}
\begin{proof}This is a direct consequence of Proposition \ref{boundpe} and Remark \ref{boundpeuhnl}. Let us mention that the predictability property of $\psi$ with values in $L^2(\Lambda)$ is inherited from $(\pe(\uhnl))_m$ at the limit.
\end{proof}
\begin{lem}\label{CVnpapp}There exist not relabeled subsequences of $\left(-\frac{(\uhnr)^-}{\epsilon}\right)_{m}$ and $\left(\frac{(\uhnr-1)^+}{\epsilon}\right)_{m}$, and $\psi_1,\psi_2$ in $L^2_{\mathcal{P}_T}\big(\Omega\times(0,T);L^2(\Lambda)\big)$ such that 
\[-\frac{(\uhnr)^-}{\epsilon}\rightarrow \psi_1\text{ and }\frac{(\uhnr-1)^+}{\epsilon}\rightarrow \psi_2,\text{ both weakly in }L^2(\Omega;L^2(0,T;\lzlambda))\text{ as }m\rightarrow+\infty.\]
\end{lem}
\begin{proof}
This is a direct consequence of Lemma \ref{buhnrppanp}, Lemma \ref{buhnlppanp} and Remark \ref{remlimcomu}.
\end{proof}
\begin{lem}\label{CVSppapn} Both strongly in $L^2(\Omega;L^2(0,T;\lzlambda))$, the following convergences hold:
\[
(\uhnr)^-\rightarrow 0 \text{ and }(\uhnr-1)^+\rightarrow 0 \text{ as }m\rightarrow+\infty.
\]
\end{lem}
\begin{proof} By Lemma \ref{buhnrppanp}, we have the existence of a constant $M>0$ independent of the regularization and discretization parameters $\epsilon\in \mathbb{R}_+^\star$, $N\in\mathbb{N}^{\star}$ and $h\in \mathbb{R}_+^\star$ such that 
\[
||(\uhnr)^-||^2_{L^2(\Omega;L^2(0,T;\lzlambda))}+||(\uhnr-1)^+||^2_{L^2(\Omega;L^2(0,T;\lzlambda))}\leq M\epsilon^2,
\]
and the announced result holds.
\end{proof}
\begin{remark}\label{cvtpsi2}[Additional informations about $\psi_1$ and $\psi_2$] \quad\\
Firstly, note that since $-\frac{(\uhnr)^-}{\epsilon}\leq 0$ and $\frac{(\uhnr-1)^+}{\epsilon}\geq 0$, then $\psi_1\leq 0$ and $\psi_2 \geq 0$ a.e. in $\Omega\times(0,T)\times\Lambda$.\\ 
Secondly, using the fact that $\pe(\uhnr)=-\frac{(\uhnr)^-}{\epsilon}+\frac{(\uhnr-1)^{+}}{\epsilon}$, one gets owing to Lemmas \ref{CVpeuhnr} and \ref{CVnpapp} that $\psi=\psi_1+\psi_2$.\\
Thirdly, since
\[
\pe(\uhnr)\uhnr=\frac{(\uhnr)^-}{\epsilon}\times (\uhnr)^-+\frac{(\uhnr-1)^+}{\epsilon}\times \left((\uhnr-1)^++1\right),
\]
one obtains thanks to Lemmas \ref{CVnpapp} and \ref{CVSppapn} that
\[
\erwb \int_0^T\int_{\Lambda}\pe(\uhnr(t,x))\uhnr(t,x)\,dx\,dt \erwe\rightarrow \erwb \int_0^T\int_{\Lambda}\psi_2(t,x)\,dx\,dt \erwe,
\]
as $m\rightarrow+\infty$.
\end{remark}

\begin{lem}\label{CVguhnl} There exists a process $g_u$ in $L^2_{\mathcal{P}_T}\big(\Omega\times(0,T); H^1(\Lambda)\big)$ such that, up to subsequences of $(g(\uhnr))_m$ and $(g(\uhnl))_m$ denoted in the same way,
\[
g(\uhnr)\rightarrow g_u \text{ and } g(\uhnl)\rightarrow g_u,\text{ both weakly in }L^2(\Omega;L^2(0,T;L^2(\Lambda)))\text{ as }m\rightarrow+\infty.
\]
\end{lem}

\begin{proof} This is mainly due to Lemma \ref{boundguhnlr}, Lemma \ref{boundsnguhnlr} and \eqref{210824_05}. A detailed proof can be found in \cite{BSZ23}.
\end{proof}

\begin{lem}\label{CVbetauhnl} There exists a process $\beta_u$ in $L^2_{\mathcal{P}_T}\big(\Omega\times(0,T); L^2(\Lambda)\big)$ such that, up to subsequences of $(\beta(\uhnr))_m$ and $(\beta(\uhnl))_m$ denoted in the same way,
\[
\beta(\uhnr)\rightarrow \beta_u \text{ and } \beta(\uhnl)\rightarrow \beta_u,\text{ both weakly in }L^2(\Omega;L^2(0,T;L^2(\Lambda)))\text{ as }m\rightarrow+\infty.
\]
\end{lem}
\begin{proof} This is a direct consequence of Lemma \ref{boundguhnlr} and \eqref{210824_05}.
\end{proof}
\begin{lem}\label{CVfhnl}
The sequence $(f^l_{h,N})_m$ defined by $(\ref{eq:notation_wh})$ and $(\ref{deffnk})$ converges strongly towards $f$ in $L^2(\Omega;L^2(0,T;L^2(\Lambda)))$ as $m\rightarrow +\infty$.
\end{lem}
\begin{proof}
There exists $\tilde{\Omega}\subset\mathcal{F}$ with $\mathds{P}({\tilde{\Omega}})=1$ such that, for all $\omega\in\tilde{\Omega}$,
by a standard argument of Steklov average, it is well-known that $\displaystyle\lim_{m\rightarrow\infty}f^l_{h,N}(\omega)= f(\omega)$ in $L^2((0,T)\times \Lambda)$.
Moreover, since $f \in L^2(\Omega; L^2((0,T)\times \Lambda))$, the mapping $\omega\mapsto ||f(\omega)||_{L^2((0,T)\times \Lambda)}$ is an element of $L^2(\Omega)$, and using the fact that $\mathbb{P}$-a.s in $\Omega$
\[
||f^l_{h,N}||_{L^2((0,T)\times\Lambda)}\leq ||f||_{L^2((0,T)\times\Lambda)},
\]
one can conclude thanks to Lebesgue's dominated convergence theorem (\cite[Theorem 1.3.3]{D}) that $(f^l_{h,N})_m$ converges strongly towards $f$ in $L^2(\Omega;L^2((0,T)\times \Lambda))$, hence, by isometry, also in $L^2(\Omega;L^2(0,T;L^2(\Lambda)))$.
\end{proof}

\begin{prop}\label{PTTL} The weak limit $u$ of our finite-volume scheme \eqref{eq:def_u0}-\eqref{equationapprox} introduced in Proposition \ref{addreg u} has $\mathds{P}$-a.s. continuous paths with values in $L^2(\Lambda)$ and satisfies for all $t\in [0,T]$,
\[u(t)=u_0+\int_0^t\left(\Delta u(s)-\psi+\beta_u(s)+f(s)\right)\,ds+\int_0^t g_u(s)\,dW(s),
\quad\]
in $L^2(\Lambda)$ and $\mathds{P}$-a.s. in $\Omega$, where $\Delta$ denotes the Laplace operator on $H^1(D)$ associated with the formal Neumann boundary conditions, $\psi$, $g_u$ and $\beta_u$ respectively are given by Lemmas \ref{CVpeuhnr}, \ref{CVguhnl} and \ref{CVbetauhnl}.
\end{prop}

\begin{proof} Let $A\in\mathcal{A}$, $\xi\in \mathscr{D}(\re)$ with $\xi(T)=0$ and $\varphi\in \mathscr{D}(\re^d)$ with $\nabla\varphi\cdot\mathbf{n}=0$ on $\partial\Lambda$, where we denote $\mathscr{D}(D):=\mathscr{C}_c^\infty(D)$ for any open subset $D\subseteq \re^d,d=2,3$. We introduce the discrete function $\varphi_h : \Lambda\rightarrow \R$ defined by $\displaystyle\varphi_h(x)=\sum_{K\in \Tau}\mathds{1}_K(x)\varphi(x_K)$ for any $x\in \Lambda$.\\
For $K\in\Tau$, $n\in\{0,\dots,N-1\}$ and $t\in[t_n,t_{n+1})$ we multiply \eqref{equationapprox} with $\mathds{1}_A\xi(t)\varphi(x_K)$ to obtain
\begin{align}\label{220814}
\begin{split}
&\mathds{1}_A\xi(t)\frac{m_K}{\dt}\Big(u_K^{n+1}-\unk-g(\unk)\big(W^{n+1}-W^n\big)\Big)\varphi(x_K)\\
&+\mathds{1}_A\xi(t)\sum_{\sigma=K|L\in\edgesint\cap\edges_K}\frac{m_\sigma}{\dkl}(u_K^{n+1}-\unpl)\varphi(x_K)\\
&+\mathds{1}_A\xi(t) m_K \pe(\unpk)\varphi(x_K)\\
=\,&\mathds{1}_A\xi(t)m_K\left(\beta(\unpk)+\fnk\right)\varphi(x_K).
\end{split}
\end{align}
Firstly, we sum \eqref{220814} over each control volume $K\in\Tau$, we integrate over the time interval $[t_{n},t_{n+1}]$, then we sum over $n=0,\dots,N-1$, and finally we take the expectation to obtain 
\begin{equation}\label{discretequforlimit}
S_{1,m}+S_{2,m}+S_{3,m}+S_{4,m}=S_{5,m}+S_{6,m}
\end{equation}
where
\begin{eqnarray*}
S_{1,m}&=&\mathbb{E}\left[\sum_{n=0}^{N-1}\int_{t_n}^{t_{n+1}}\sum_{K\in\Tau}\mathds{1}_A\xi(t)m_K\frac{u_K^{n+1}-\unk}{\dt}\varphi(x_K)\,dt\right]\\
S_{2,m}&=&-\mathbb{E}\left[\sum_{n=0}^{N-1}\int_{t_n}^{t_{n+1}}\sum_{K\in\Tau}\mathds{1}_A\xi(t)m_Kg(\unk)\frac{W^{n+1}-W^n}{\dt}\varphi(x_K)\,dt\right]\\
S_{3,m}&=&\mathbb{E}\left[\sum_{n=0}^{N-1}\int_{t_n}^{t_{n+1}}\sum_{K\in\Tau}\mathds{1}_A\xi(t)\sum_{\sigma=K|L\in\edgesint\cap\edges_K}\frac{m_\sigma}{\dkl}(u_K^{n+1}-\unpl)\varphi(x_K)\,dt\right]\\
S_{4,m}&=&\mathbb{E}\left[\sum_{n=0}^{N-1}\int_{t_n}^{t_{n+1}}\sum_{K\in\Tau}\mathds{1}_A\xi(t) m_K\pe(\unpk)
\varphi(x_K)\,dt\right]\\
S_{5,m}&=&\mathbb{E}\left[\sum_{n=0}^{N-1}\int_{t_n}^{t_{n+1}}\sum_{K\in\Tau}\mathds{1}_A\xi(t)m_K \beta(\unpk)\varphi(x_K)\,dt\right]\\
S_{6,m}&=&\mathbb{E}\left[\sum_{n=0}^{N-1}\int_{t_n}^{t_{n+1}}\sum_{K\in\Tau}\mathds{1}_A\xi(t)m_K \fnk\varphi(x_K)\,dt\right].
\end{eqnarray*}
Let us study separately the limit as $m$ goes to $+\infty$ of $S_{1,m}$, $S_{2,m}$, $S_{3,m}$, $S_{4,m}$, $S_{5,m}$ and $S_{6,m}$.\medskip\\
$\bullet$ Study of $S_{1,m}$: Following \cite[Proposition 4.5]{BSZ23}, one proves thanks to Proposition \ref{addreg u} and a discrete integration by parts formula, that up to a subsequence denoted in the same way 
\begin{align*}
S_{1,m}\xrightarrow[m\rightarrow +\infty]{}-\mathbb{E}\left[\mathds{1}_A\int_0^T\int_\Lambda u(t,x)\xi'(t)\varphi(x)\,dx\,dt\right]-\mathbb{E}\left[\mathds{1}_A\int_\Lambda u_0(x)\xi(0)\varphi(x)\,dx\right].
\end{align*}
$\bullet$ Study of $S_{2,m}$: 
Thanks to Lemma \ref{CVguhnl} and the properties of the stochastic integral, one shows that, up to a subsequence denoted in the same way (see \cite[Proposition 4.5]{BSZ23})
\begin{align*}
S_{2,m}\xrightarrow[m\rightarrow +\infty]{}\mathbb{E}\left[\mathds{1}_A\int_0^T\int_\Lambda \int_0^t g_u(s,x)\,dW(s)\xi'(t)\varphi(x)\,dx\,dt\right].
\end{align*}
$\bullet$ Study of $S_{3,m}$: 
Following the arguments we developed in \cite[Proposition 4.16]{BNSZ22}, one shows that 
\begin{align*}
S_{3,m}\xrightarrow[m\rightarrow +\infty]{} -\mathbb{E}\left[\mathds{1}_A\int_0^T\int_\Lambda\xi(t)\Delta\varphi(x)u(t,x)\,dx\,dt\right].
\end{align*}
$\bullet$ Study of $S_{4,m}$: Using Lemma \ref{CVpeuhnr}, one proves that
\[
S_{4,m}\xrightarrow[m\rightarrow +\infty]{} \mathbb{E}\left[\mathds{1}_A\int_0^T \int_{\Lambda} \psi(t,x) \varphi(x)\xi(t)\,dx\,dt\right].
\]
To do so, we use the following decomposition for $S_{4,m}=S_{4,m}-\tilde{S}_{4,m}+\tilde{S}_{4,m}$, where
\begin{align*}
\tilde{S}_{4,m}=\mathbb{E}\left[\mathds{1}_A\int_0^T\int_{\Lambda} \pe(\uhnr)\varphi(x)\xi(t)\,dx\,dt\right]\xrightarrow[m\rightarrow +\infty]{}\mathbb{E}\left[\mathds{1}_A\int_0^T \int_{\Lambda} \psi(t,x) \varphi(x)\xi(t)\,dx\,dt\right].
\end{align*}
Note that 
\begin{align*}
|S_{4,m}-\tilde{S}_{4,m}|&=&\left|\mathbb{E}\left[\mathds{1}_A\sum_{n=0}^{N-1}\sum_{K\in\Tau}\int_{t_n}^{t_{n+1}}\int_{K} \pe(\unpk)\big(\varphi(x_K)-\varphi(x)\big)\xi(t)\,dx\,dt\right]\right|\\
&\leq& h||\xi||_{\infty} ||\nabla \varphi||_{\infty} ||\pe(\uhnr) ||_{L^1\left(\Omega;L^1(0,T;L^1(\Lambda))\right)}\xrightarrow[m\rightarrow +\infty]{} 0.
\end{align*}
$\bullet$ Study of $S_{5,m}$: 
Thanks to Lemma \ref{CVbetauhnl}, one shows as for the study of $S_{4,m}$ that 
\begin{align*}
S_{5,m}\xrightarrow[m\rightarrow +\infty]{} \mathbb{E}\left[\mathds{1}_A\int_0^T \int_{\Lambda} \beta_u(t,x)\varphi(x)\xi(t)\,dx\,dt\right].
\end{align*}
$\bullet$ Study of $S_{6,m}$: using Lemma \ref{CVfhnl} and the fact that $f$ also belongs to $L^1(\Omega;L^1(0,T;L^1(\Lambda)))$,
one proves that
\begin{align*}
S_{6,m}\xrightarrow[m\rightarrow +\infty]{} \mathbb{E}\left[\mathds{1}_A\int_0^T \int_{\Lambda} f(t,x)\varphi(x)\xi(t)\,dx\,dt\right].
\end{align*}
Indeed,
\begin{align*}
&\left|S_{6,m}-\mathbb{E}\left[\mathds{1}_A\int_0^T \int_{\Lambda} f(t,x)\varphi(x)\xi(t)\,dx\,dt\right]\right|\\
=&\left|\mathbb{E}\left[\sum_{n=0}^{N-1}\int_{t_n}^{t_{n+1}}\sum_{K\in\Tau}\mathds{1}_A\xi(t)m_K \fnk\varphi(x_K)\,dt\right]-\mathbb{E}\left[\mathds{1}_A\int_0^T \int_{\Lambda} f(t,x)\varphi(x)\xi(t)\,dx\,dt\right]\right|\\
=&\left|\mathbb{E}\left[\mathds{1}_A\sum_{n=0}^{N-1}\sum_{K\in\Tau}\int_{t_n}^{t_{n+1}}\int_K\xi(t)\left\{\fnk\varphi(x_K)-f(t,x)\varphi(x)\right\}\,dx\,dt\right]\right|\\
=&\left|\mathbb{E}\left[\mathds{1}_A\sum_{n=0}^{N-1}\sum_{K\in\Tau}\int_{t_n}^{t_{n+1}}\int_K\xi(t)\Big\{\varphi(x_K)\big(\fnk-f(t,x)\big)+f(t,x)\big(\varphi(x_K)-\varphi(x)\big)\Big\}\,dx\,dt\right]\right|\\
\leq& ||\xi||_{\infty} ||\varphi||_{\infty} ||f^l_{h,N}-f||_{L^1\left(\Omega;L^1((0,T)\times \Lambda)\right)}+h||\xi||_{\infty}||\nabla\varphi||_{\infty} ||f||_{L^1(\Omega;L^1(0,T;L^1(\Lambda)))}\xrightarrow[m\rightarrow+\infty]{}0.
\end{align*}
Gathering all the previous convergence results, we can pass to the limit in \eqref{discretequforlimit}, and, by using the density of the set $\{\Psi\in\mathscr{D}(\mathbb{R}^d)\ | \ \nabla\Psi\cdot\mathbf{n}=0 \ \text{on} \ \partial\Lambda\}$ in $H^1(\Lambda)$ (given by \cite[Theorem 1.1]{Droniou}), we get that $\mathds{P}$-a.s. in $\Omega$, for all $\xi\in \big\{\phi\in\mathscr{D}(\re): \phi(T)=0\big\}$ and all $\varphi\in H^1(\Lambda)$
\begin{align}\label{210830_08bis}
\begin{split}
&-\int_0^T\int_\Lambda\left(u(t,x)-\int_0^{t}g_u(s,x)\,dW(s)\right)\xi'(t)\varphi(x)\,dx\,dt-\int_\Lambda u_0(x)\xi(0)\varphi(x)\,dx\\
=&-\int_0^T \int_\Lambda \nabla u(t,x) \cdot \nabla \varphi(x)\xi(t)\,dx\,dt- \int_0^T\int_\Lambda \psi(t,x)\varphi(x)\xi(t)\,dx\,dt\\
&+\int_0^T\int_\Lambda \beta_u(t,x)\varphi(x)\xi(t)\,dx\,dt+\int_0^T\int_\Lambda f(t,x)\varphi(x)\xi(t)\,dx\,dt.
\end{split}
\end{align}
By identically repeating the arguments developed in the proof of \cite[Proposition 4.5]{BSZ23}, we first obtain that $u\in L^2\big(\Omega; \mathscr{C}([0,T];L^2(\Lambda))\big)$, and then that for any $t$ in $[0,T]$
\begin{align*}
u(t)-u(0)-\int_0^t g_u(s)\,dW(s)+\int_0^t \psi(s)ds- \int_0^t\beta_u(s)\,ds- \int_0^tf(s)\,ds=\int_0^t\Delta u(s)\,ds,
\end{align*}
in $H^1(\Lambda)^{\ast}$ and $\mathds{P}$-a.s. in $\Omega$. To conclude, let us mention that since the left-hand side of the above equality is in $L^2(\Lambda)$, it also holds in $L^2(\Lambda)$.
\end{proof}

\begin{lem}{(Stochastic energy equality)
}\label{lemenergy}
For any $c>0$, the stochastic process $u$ introduced in Proposition \ref{addreg u} satisfies the following stochastic energy equality: 
\begin{align}\label{energy}
\begin{split}
&e^{-ct} \E\left[||u(t) ||^2_{L^2(\Lambda)}\right]+2\int_0^te^{-cs} \E\left[||\nabla u(s) ||^2_{L^2(\Lambda)}\right]\,ds\\
=\,& \E\left[|| u_0||^2_{L^2(\Lambda)}\right]-c\int_0^t e^{-cs} \E\left[|| u(s)||^2_{L^2(\Lambda)}\right]\,ds+\int_0^t e^{-cs} \E\left[||g_u(s)||^2_{L^2(\Lambda)}\right]\,ds\\
&+2\int_{0}^te^{-cs} \E\left[ \int_{\Lambda}\big(\beta_u(s,x)+f(s,x)-\psi(s,x)\big)u(s,x)\,dx \right]\,ds, \quad \forall t\in [0,T].
\end{split}
\end{align}
\end{lem}
\begin{proof} It is a direct application of It\^o formula to the stochastic process $u$ and the functional $\mathscr{F}: (t,v)\mapsto e^{-ct}||v||^2_{L^2(\Lambda)}$ defined on $[0,T]\times L^2(\Lambda)$. 
\end{proof}
\subsection{Identification of weak limits coming from the non-linear terms}
We state here a result proved in \cite[Lemma 4.7]{BSZ23}, which gives a lower bound for the inferior limit of the following quantity
\begin{align*}
\int_0^T\int_{0}^{t}e^{-cs}\E[|u^r_{h_m,N_m}(s)|^2_{1,h_m}]\,ds\,dt,
\end{align*}
for any $c>0$. This boundedness result will be one of the key points in identifying the weak limits $\psi$, $g_u$, $\beta_u$ coming from the discretization of the non linear terms $\pe(\uhnr), g(\uhnr)$ and $\beta(\uhnr)$, see Lemmas \ref{CVpeuhnr}, \ref{CVguhnl}, \ref{CVbetauhnl}.
\begin{lem}\label{keylemma} For any $c>0$, the stochastic process $u$ introduced in Proposition \ref{addreg u} satisfies the following inequality:
\begin{align}\label{liminfinequality}
\hspace*{-0.4cm}\int_0^T\int_0^t e^{-cs}\E\left[\int_{\Lambda}|\nabla u(x,s) |^2\,dx\right]\,ds\,dt\leq\liminf_{m\rightarrow+\infty}\int_0^T\int_{0}^{t}e^{-cs}\E[|u^r_{h_m,N_m}(s)|^2_{1,h_m}]\,ds\,dt.
\end{align}
\end{lem}
Now, we have all the necessary tools on the one hand for the identification of $\psi$, $\psi_1$, $\psi_2$, $g_u$ and $\beta_u$, and on the other hand for completing the proof of Theorem \ref{mainresult}.

\begin{prop}\label{PISL}
The sequences $(\uhnr)_m$ and $(\uhnl)_m$ converge strongly in \linebreak $L^2(\Omega;L^2(0,T;L^2(\Lambda)))$ to the unique variational solution of Problem \eqref{equation} in the sense of Definition \ref{solution}.
\end{prop}
\begin{proof}
Let us fix $n\in \{0, ..., N-1\}$, $K\in \Tau$, and multiply \eqref{equationapprox}
by $\dt u_K^{n+1}$, use the formula $a(a-b)=\frac{1}{2}(a^2-b^2+(a-b)^2)$ with $a=u_K^{n+1}$ and $b=\unk$, take the expectation, and proceed as for the obtention of \eqref{term3} to arrive at
\begin{eqnarray*}
&&\frac{m_K}{2}\E\left[(u_K^{n+1})^2-(\unk)^2\right]+\frac{m_K}{2}\E\left[(u_K^{n+1}-\unk)^2\right]\\
&&+\dt\sum_{\sigma=K|L\in\edgesint\cap\edges_K}\frac{m_\sigma}{\dkl}\E\left[(u_K^{n+1}-\unpl)u_K^{n+1}\right]+\dt\, m_K\pe(\unpk)\unpk\\
&\leq&\frac{m_K}{2}\E\left[(u_K^{n+1}-\unk)^2\right]+\frac{m_K\,\dt}{2}\E\left[g^2(\unk)\right]+\dt\, m_K \E\left[\big(\beta(\unpk)+\fnk\big)\unpk\right].
\end{eqnarray*}
Now, we multiply the last inequality by $e^{-c\tn}$ for arbitrary $c>0$. Then, summing over $K\in \Tau$ and $n\in\{0, ..., k\}$ for $k\in\{0, ..., N-1\}$, using \eqref{PInt} and reasoning as in the proof of \eqref{term2} one gets
\begin{eqnarray*}
&&\frac{1}{2}\sum_{n=0}^{k}\sum_{K\in \Tau}m_Ke^{-c\tn}\E\left[(u_K^{n+1})^2-(\unk)^2\right]+\dt\sum_{n=0}^{k}e^{-c\tn}\hspace*{-0.4cm}\sum_{\sigma=K|L\in\edgesint}\frac{m_\sigma}{\dkl}\E\left[|u_K^{n+1}-\unpl|^2\right]\\
&&+\dt\sum_{n=0}^{k}\sum_{K\in \Tau}m_Ke^{-c\tn}\E\left[\pe(\unpk)\unpk\right]\\
&\leq&\frac{\dt}{2}\sum_{n=0}^{k}\sum_{K\in \Tau}m_Ke^{-c\tn}\E\left[g^2(\unk)\right]+\dt\sum_{n=0}^{k}\sum_{K\in \Tau}m_Ke^{-c\tn}\E\left[\big(\beta(\unpk)+\fnk\big)\unpk\right].
\end{eqnarray*}
Let us focus on each sum of this last inequality separately, by using the computations we developed in \cite{BSZ23}.\\
$\bullet$ Note that the general term of the first sum can be decomposed in the following way: 
\begin{eqnarray*}
&&e^{-c\tn}\E\left[(u_K^{n+1})^2-(\unk)^2\right]\\
&=&e^{-c\tn}\E\left[(u_K^{n+1})^2\right]-e^{-ct_{n-1}}\E\left[(\unk)^2\right]-\E\left[(\unk)^2\right]\left( e^{-c\tn}-e^{-ct_{n-1}} \right),
\end{eqnarray*}
where $t_{-1}:=-\dt$.
Firstly, we have
\begin{align}\label{terme1a}
\begin{split}
&\frac{1}{2}\sum_{n=0}^{k}\sum_{K\in \Tau}m_K \left(e^{-c\tn}\E\left[(u_K^{n+1})^2\right]-e^{-ct_{n-1}}\E\left[(\unk)^2\right]\right)\\
=\,&\frac{1}{2}\sum_{K\in \Tau} m_Ke^{-ct_k}\E\left[(\ukpk)^2\right]-\frac{1}{2}\sum_{K\in \Tau} m_K\E\left[(u_K^{0})^2\right]e^{c\dt}.
\end{split}
\end{align}
Then, using properties of the exponential function, 
\begin{align}\label{terme1b}
\begin{split}
&-\frac{1}{2}\sum_{n=0}^{k}\sum_{K\in \Tau}m_K \E\left[(\unk)^2\right]\left( e^{-c\tn}-e^{-ct_{n-1}} \right)\\
>&-\frac{1}{2}\sum_{K\in \Tau}m_K \E\left[(u^0_K)^2\right]\left( 1-e^{c\dt} \right) +\frac{c}{2}e^{-c\dt}\int_0^{t_k}e^{-cs}\E\left[||u^r_{h,N}(s)||^2_{L^2(\Lambda)}\right]ds\\
>&\ \frac{c}{2}e^{-c\dt}\int_0^{t_k}e^{-cs}\E\left[||u^r_{h,N}(s)||^2_{L^2(\Lambda)}\right]ds.
\end{split}
\end{align}
$\bullet$ The second sum can be handled similarly in the following manner
\begin{align}\label{terme2}
\begin{split}
\dt\sum_{n=0}^{k}e^{-c\tn}\sum_{\sigma=K|L\in\edgesint}\frac{m_\sigma}{\dkl}\E\left[|u_K^{n+1}-\unpl|^2\right]
=&\ \dt\sum_{n=0}^{k}e^{-c\tn}\E[|u_h^{n+1}|^2_{1,h}]\\
\geq&\ \int_{0}^{t_{k+1}}e^{-cs}\E[|u^r_{h,N}(s)|^2_{1,h}]\,ds.
\end{split}
\end{align}
$\bullet$ Since $\pe$ is non-decreasing and satisfies $\pe(0)=0$, the third sum can be handled as follows:
\begin{align}\label{terme3}
\begin{aligned}
&\int_{0}^{t_{k+1}}e^{-cs}\E\left[\int_{\Lambda}\pe(u^r_{h,N}(s,x))u^r_{h,N}(s,x)\,dx\right]\,ds\\
=&\sum_{n=0}^{k}\sum_{K\in \Tau}\int_{\tn}^{\tnp}e^{-cs}\E\left[\int_K\pe(\unpk)\unpk\,dx\right]\,ds\\
\leq&\dt\sum_{n=0}^{k}\sum_{K\in \Tau}m_Ke^{-c\tn}\E\left[\pe(\unpk)\unpk\right].
\end{aligned}
\end{align}
$\bullet$ We have also the following majoration of the fourth sum: 
\begin{align}\label{terme4}
\begin{split}
&\frac{\dt}{2}\sum_{n=0}^{k}\sum_{K\in \Tau}m_Ke^{-c\tn}\E\left[g^2(\unk)\right]\\
\leq\,&\frac{\dt}{2}\sum_{K\in \Tau}m_K\E\left[g^2(u_K^0)\right]+\frac{1}{2}\int_{0}^{t_k} e^{-cs}\E\left[||g(u^r_{h,N})(s)||^2_{L^2(\Lambda)}\right]\,ds.
\end{split}
\end{align}
$\bullet$ Using the properties of the exponential function again, the last sum can be handled in the following manner: 
\begin{align}\label{terme5}
\begin{aligned}
&\dt\sum_{n=0}^{k}\sum_{K\in \Tau}m_Ke^{-c\tn}\E\left[\big(\beta(\unpk)+\fnk\big)\unpk\right]\\
\leq&\int_{0}^{t_{k+1}}e^{-cs}\E\left[\int_{\Lambda}\big(\beta(\uhnr(s,x))+f_{h,N}^l(s,x)\big)\uhnr(s,x)\,dx\right]\,ds\\
&+c\dt ||\uhnr||_{L^2\left(\Omega;L^2(0,T;L^2(\Lambda))\right)} \left(L_{\beta}||\uhnr||_{L^2\left(\Omega;L^2(0,T;L^2(\Lambda))\right)} +||f||_{L^2\left(\Omega;L^2(0,T;L^2(\Lambda))\right)}\right).
\end{aligned}
\end{align}
Combining \eqref{terme1a}, \eqref{terme1b}, \eqref{terme2}, \eqref{terme3}, \eqref{terme4} and \eqref{terme5}, one gets
\begin{eqnarray*}
&&\sum_{K\in \Tau} m_Ke^{-ct_k}\E\left[(\ukpk)^2\right]+2\int_{0}^{t_{k+1}}e^{-cs}\E[|u^r_{h,N}(s)|^2_{1,h}]\,ds\\
&&+2\int_{0}^{t_{k+1}}e^{-cs}\E\left[\int_{\Lambda}\pe(u^r_{h,N}(s,x))u^r_{h,N}(s,x)\,dx\right]\,ds\\
&\leq&e^{c\dt}\sum_{K\in \Tau} m_K\E\left[(u_K^{0})^2\right]+\dt\sum_{K\in \Tau}m_K\E\left[g^2(u_K^0)\right]+\int_{0}^{t_k} e^{-cs}\E\left[||g(u^r_{h,N})(s)||^2_{L^2(\Lambda)}\right]ds\\
&&+2\int_{0}^{t_{k+1}}e^{-cs}\E\left[\int_{\Lambda}\big(\beta(\uhnr(s,x))+f^l_{h,N}(s,x)\big)\uhnr(s,x)\,dx\right]\,ds\\
&&-ce^{-c\dt}\int_0^{t_k}e^{-cs}\E\left[||u^r_{h,N}(s)||^2_{L^2(\Lambda)}\right]ds\\
&&+2c\dt ||\uhnr||_{L^2\left(\Omega;L^2(0,T;L^2(\Lambda))\right)} \left(L_{\beta}||\uhnr||_{L^2\left(\Omega;L^2(0,T;L^2(\Lambda))\right)} +||f||_{L^2\left(\Omega;L^2(0,T;L^2(\Lambda))\right)}\right).
\end{eqnarray*}
Moreover, for $t\in [t_k, t_{k+1})$ since $e^{-ct}\leq e^{-ct_k}$ and $(t-\dt)^+\leq t_k$, one obtains that
\begin{eqnarray*}
&&e^{-ct}\E\left[||\uhnr(t)||^2_{L^2(\Lambda)}\right]+2\int_{0}^{t}e^{-cs}\E[|u^r_{h,N}(s)|^2_{1,h}]\,ds\\
&&+2\int_{0}^{t}e^{-cs}\E\left[\int_{\Lambda}\pe(u^r_{h,N}(s,x))u^r_{h,N}(s,x)\,dx\right]\,ds\\
&\leq&e^{c\dt}\sum_{K\in \Tau} m_K\E\left[(u_K^{0})^2\right]+\dt\sum_{K\in \Tau}m_K\E\left[g^2(u_K^0)\right]+\int_{0}^{t} e^{-cs}\E\left[||g(u^r_{h,N})(s)||^2_{L^2(\Lambda)}\right]\,ds\\
&&+2\int_{0}^{t}e^{-cs}\E\left[\int_{\Lambda}\big(\beta(\uhnr(s,x))+f^l_{h,N}(s,x)\big)\uhnr(s,x)\,dx\right]\,ds\\
&&+2\int_{t}^{t_{k+1}}e^{-cs}\E\left[\int_{\Lambda}\big(\beta(\uhnr(s,x))+f^l_{h,N}(s,x)\big)\uhnr(s,x)\,dx\right]\,ds\\
&&-ce^{-c\dt}\int_0^{(t-\dt)^+}e^{-cs}\E\left[||u^r_{h,N}(s)||^2_{L^2(\Lambda)}\right]\,ds\\
&&+2c\dt ||\uhnr||_{L^2\left(\Omega;L^2(0,T;L^2(\Lambda))\right)} \left(L_{\beta}||\uhnr||_{L^2\left(\Omega;L^2(0,T;L^2(\Lambda))\right)} +||f||_{L^2\left(\Omega;L^2(0,T;L^2(\Lambda))\right)}\right).
\end{eqnarray*}
Using the constant $K_0>0$ given by Proposition~\ref{bounds}, one gets the estimate
\begin{align*}
&\int_{(t-\dt)^+}^te^{-cs}\E\left[||u^r_{h,N}(s)||^2_{L^2(\Lambda)}\right]ds\\
&+2\int_{t}^{t_{k+1}}e^{-cs}\E\left[\int_{\Lambda}\big(\beta(\uhnr(s,x))+f^l_{h,N}(s,x)\big)\uhnr(s,x)\,dx\right]\,ds
\\
\leq&\dt (1+2L_{\beta})K_0+2\sqrt{\dt}||f||_{L^2\left(\Omega;L^2(0,T;L^2(\Lambda))\right)}\sqrt{K_0},
\end{align*}
and Chasles' relation $-\int_0^{(t-\dt)^+}=-\int_0^t+\int_{(t-\dt)^+}^t$ yields
\begin{align}\label{ineginter}
	&e^{-ct}\E\left[||\uhnr(t)||^2_{L^2(\Lambda)}\right]+2\int_{0}^{t}e^{-cs}\E[|u^r_{h,N}(s)|^2_{1,h}]\,ds\nonumber \\
	&+2\int_{0}^{t}e^{-cs}\E\left[\int_{\Lambda}\pe(u^r_{h,N}(s,x))u^r_{h,N}(s,x)\,dx\right]\,ds\nonumber\\
	\leq\,&e^{c\dt}\E[||u_0||^2_{L^2(\Lambda)}]+\dt L_g^2\E[||u_0||^2_{L^2(\Lambda)}]+\int_{0}^{t} e^{-cs}\E\left[||g(u^r_{h,N})(s)||^2_{L^2(\Lambda)}\right]\,ds\\
	&+2\int_{0}^{t}e^{-cs}\E\left[\int_{\Lambda}\big(\beta(\uhnr(s,x))+f^l_{h,N}(s,x)\big)\uhnr(s,x)\,dx\right]\,ds\nonumber\\
	&-ce^{-c\dt}\int_0^{t}e^{-cs}\E\left[||u^r_{h,N}(s)||^2_{L^2(\Lambda)}\right]\,ds\nonumber\\
	&+c\dt (1+2L_{\beta})K_0+2c\sqrt{\dt}||f||_{L^2\left(\Omega;L^2(0,T;L^2(\Lambda))\right)}\sqrt{K_0}\nonumber\\
	&+2c\dt ||\uhnr||_{L^2\left(\Omega;L^2(0,T;L^2(\Lambda))\right)} \left(L_{\beta}||\uhnr||_{L^2\left(\Omega;L^2(0,T;L^2(\Lambda))\right)} +||f||_{L^2\left(\Omega;L^2(0,T;L^2(\Lambda))\right)}\right).\nonumber
\end{align}
Furthermore,
\begin{align}\label{decompog}
\begin{split}
&\int_{0}^{t} e^{-cs}\E\left[||g(u^r_{h,N})(s)||^2_{L^2(\Lambda)}\right]ds\\
=&\int_{0}^{t} e^{-cs}\E\left[||g(u^r_{h,N})(s)-g(u)(s)||^2_{L^2(\Lambda)}\right]ds\\
&+2\int_{0}^{t} e^{-cs}\E\left[\int_{\Lambda}g(u^r_{h,N})(s,x)g(u)(s,x)\,dx\right]\,ds-\int_{0}^{t} e^{-cs}\E\left[||g(u)(s)||^2_{L^2(\Lambda)}\right]ds.
\end{split}
\end{align}
In the same manner,
\begin{align}\label{decompou}
&-ce^{-c\dt}\int_0^{t}e^{-cs}\E\left[||u^r_{h,N}(s)||^2_{L^2(\Lambda)}\right]ds\nonumber\\
=&-ce^{-c\dt}\int_{0}^{t} e^{-cs}\E\left[||u^r_{h,N}(s)-u(s)||^2_{L^2(\Lambda)}\right]ds\\
&-2ce^{-c\dt}\int_{0}^{t} e^{-cs}\E\left[\int_{\Lambda}u^r_{h,N}(s,x)u(s,x)\,dx\right]\,ds+ce^{-c\dt}\int_{0}^{t} e^{-cs}\E\left[||u(s)||^2_{L^2(\Lambda)}\right]ds.\nonumber
\end{align}
And at last
\begin{align}\label{decompobeta}
&\int_{0}^{t}e^{-cs}\E\left[\int_{\Lambda}\beta(\uhnr(s,x))\uhnr(s,x)\,dx\right]\,ds\nonumber\\
=&\int_{0}^{t}e^{-cs}\E\left[\int_{\Lambda}\beta(\uhnr(s,x))u(s,x)\,dx\right]\,ds\nonumber\\
&+\int_{0}^{t}e^{-cs}\E\left[\int_{\Lambda}\big(\beta(\uhnr(s,x))-\beta(u(s,x))\big)(\uhnr(s,x)-u(s,x))\,dx\right]\,ds\\
&+\int_{0}^{t}e^{-cs}\E\left[\int_{\Lambda}\beta(u(s,x))(\uhnr(s,x)-u(s,x))\,dx\right]\,ds.\nonumber
\end{align}
Finally, by considering from now on a parameter $c>0$ depending only on $L_g$ and $L_{\beta}$ such that for any $N$ big enough
\begin{align*}
&\int_{0}^{t} e^{-cs}\E\left[||g(u^r_{h,N}(s))-g(u(s))||^2_{L^2(\Lambda)}\right]ds-ce^{-c\dt}\int_{0}^{t} e^{-cs}\E\left[||u^r_{h,N}(s)-u(s)||^2_{L^2(\Lambda)}\right]ds\\
&+2\int_{0}^{t}e^{-cs}\E\left[\int_{\Lambda}\big(\beta(\uhnr(s,x))-\beta(u(s,x))\big)(\uhnr(s,x)-u(s,x))\,dx\right]\,ds\leq 0,
\end{align*}
we are able to prove that, after injecting \eqref{decompog}, \eqref{decompou} and \eqref{decompobeta} in \eqref{ineginter}, and integrating from $0$ to $T$:
\begin{align*}
&\int_0^Te^{-ct}\E\left[||\uhnr(t)||^2_{L^2(\Lambda)}\right]\,dt+2\int_0^T\int_{0}^{t}e^{-cs}\E[|u^r_{h,N}(s)|^2_{1,h}]\,ds\,dt\\
&+2\int_0^T\int_{0}^{t}e^{-cs}\E\left[\int_{\Lambda}\pe(u^r_{h,N}(s,x))u^r_{h,N}(s,x)\,dx\right]\,ds\,dt\nonumber\\
\leq\,&\int_0^T\E[||u_0||^2_{L^2(\Lambda)}]\,dt+2\int_0^T\int_{0}^{t} e^{-cs}\E\left[\int_{\Lambda}g(u^r_{h,N})(s,x)g(u)(s,x)\,dx\right]\,ds\,dt\\
&-\int_0^T\int_{0}^{t} e^{-cs}\E\left[||g(u(s))||^2_{L^2(\Lambda)}\right]\,ds\,dt-2ce^{-c\dt}\int_0^T\int_{0}^{t} e^{-cs}\E\left[\int_{\Lambda}u^r_{h,N}(s,x)u(s,x)\,dx\right]\,ds\,dt\\
&+ce^{-c\dt}\int_0^T\int_{0}^{t} e^{-cs}\E\left[||u(s)||^2_{L^2(\Lambda)}\right]\,ds\,dt+T\left(e^{c\dt}-1\right)\E[||u_0||^2_{L^2(\Lambda)}]\\
&+2\int_{0}^{T}\int_{0}^{t}e^{-cs}\E\left[\int_{\Lambda}\beta(\uhnr(s,x))u(s,x)\,dx\right]\,ds\,dt\\
&+2\int_{0}^{T}\int_{0}^{t}e^{-cs}\E\left[\int_{\Lambda}\beta(u(s,x))\big(\uhnr(s,x)-u(s,x)\big)\,dx\right]\,ds\,dt\\
&+2\int_{0}^{T}\int_{0}^{t}e^{-cs}\E\left[\int_{\Lambda}f^l_{h,N}(s,x)\uhnr(s,x)\,dx\right]\,ds\,dt\nonumber\\
&+c\dt T (1+2L_{\beta})K_0+2cT\sqrt{\dt}||f||_{L^2\left(\Omega;L^2(0,T;L^2(\Lambda))\right)}\sqrt{K_0}\nonumber\\
&+2c\dt T ||\uhnr||_{L^2\left(\Omega;L^2(0,T;L^2(\Lambda))\right)} \left(L_{\beta}||\uhnr||_{L^2\left(\Omega;L^2(0,T;L^2(\Lambda))\right)} +||f||_{L^2\left(\Omega;L^2(0,T;L^2(\Lambda))\right)}\right).\nonumber
\end{align*}
Firstly, by passing to the superior limit in this last inequality, Remark \ref{cvtpsi2} and Lemma \ref{CVfhnl} allow us to state that:
\begin{align*}
&\limsup_{m\rightarrow+\infty}\int_0^Te^{-ct}\E\left[||\uhnr(t)||^2_{L^2(\Lambda)}\right]\,dt+2\liminf_{m\rightarrow+\infty}\int_0^T\int_{0}^{t}e^{-cs}\E[|u^r_{h,N}(s)|^2_{1,h}]\,ds\,dt\\
&+2\int_0^T\int_{0}^te^{-cs}\erwb \int_{\Lambda}\psi_2(s,x)\,dx\erwe \,ds\,dt \\
\leq&\int_0^T\E[||u_0||^2_{L^2(\Lambda)}]\,dt+2\int_0^T\int_{0}^{t} e^{-cs}\E\left[\int_{\Lambda}g_u(s,x)g(u)(s,x)\,dx\right]\,ds\,dt\\
&-\int_0^T\int_{0}^{t} e^{-cs}\E\left[||g(u(s))||^2_{L^2(\Lambda)}\right]\,ds\,dt-c\int_0^T\int_{0}^{t} e^{-cs}\E\left[||u(s)||^2_{L^2(\Lambda)}\right]\,ds\,dt\\
&+2\int_{0}^{T}\int_{0}^{t}e^{-cs}\E\left[\int_{\Lambda}\beta_u(s,x)u(s,x)\,dx\right]\,ds\,dt+2\int_{0}^{T}\int_{0}^{t}e^{-cs}\E\left[\int_{\Lambda}f(s,x)u(s,x)\,dx\right]\,ds\,dt.\nonumber
\end{align*}
Secondly, the stochastic energy equality \eqref{energy} yields
\begin{eqnarray*}
&&\limsup_{m\rightarrow+\infty}\int_0^Te^{-ct}\E\left[||\uhnr(t)||^2_{L^2(\Lambda)}\right]\,dt+2\liminf_{m\rightarrow+\infty}\int_0^T\int_{0}^{t}e^{-cs}\E[|u^r_{h,N}(s)|^2_{1,h}]\,ds\,dt\\
&&+2\int_0^T\int_{0}^te^{-cs}\erwb \int_{\Lambda}\psi_2(s,x)-\psi(s,x)u(s,x)\,dx\erwe \,ds\,dt \\
&\leq&\int_0^T e^{-ct} \E\left[||u(t) ||^2_{L^2(\Lambda)}\right]\,dt +2\int_0^T\int_0^te^{-cs} \E\left[||\nabla u(s) ||^2_{L^2(\Lambda)}\right]\,ds\,dt\\
&&-\int_0^T\int_0^t e^{-cs} \E\left[||g(u(s))-g_u(s)||^2_{L^2(\Lambda)}\right]\,ds\,dt.
\end{eqnarray*}
Thirdly, thanks to the key Inequality \eqref{liminfinequality} given by Lemma \ref{keylemma}, we arrive at
\begin{eqnarray}\label{lastineq}
&&\limsup_{m\rightarrow+\infty}\int_0^Te^{-ct}\E\left[||\uhnr(t)||^2_{L^2(\Lambda)}\right]\,dt+\int_0^T\int_0^t e^{-cs} \E\left[||g(u(s))-g_u(s) ||^2_{L^2(\Lambda)}\right]\,ds\,dt\nonumber
\\
&&+2\int_0^T\int_{0}^te^{-cs}\erwb \int_{\Lambda}\psi_2(s,x)-\psi(s,x)u(s,x)\,dx\erwe \,ds\,dt \nonumber\\
&\leq&\int_0^T e^{-ct} \E\left[||u(t) ||^2_{L^2(\Lambda)}\right]\,dt.
\end{eqnarray}
\noindent Owing to the weak convergence of $(\uhnr)_{m}$ towards $u$ in $L^2(\Omega;L^2(0,T;L^2(\Lambda)))$, we can affirm that the following inequality holds true
\[
\int_0^T e^{-ct} \E\left[||u(t) ||^2_{L^2(\Lambda)}\right]\,dt\leq \liminf_{m\rightarrow+\infty}\int_0^Te^{-ct}\E\left[||\uhnr(t)||^2_{L^2(\Lambda)}\right]\,dt,
\]
so that
\begin{eqnarray}\label{keyineq} 
\hspace*{-0.5cm}\int_0^T\int_{0}^te^{-cs}\erwb \int_{\Lambda}\psi_2(s,x)-\psi(s,x)u(s,x)\,+\left(g(u(s,x))-g_u(s,x)\right)^2\,dx\erwe \,ds\,dt\leq 0.
\end{eqnarray}
By Lemma \ref{CVSppapn} and Remark \ref{cvtpsi2}, we can affirm that $\mathbb{P}$-almost surely in $\Omega$ and almost everywhere in $(0,T)\times \Lambda$, $0\leq u \leq 1$ and $\psi=\psi_1+\psi_2$ with $\psi_1\leq 0$ and $\psi_2\geq 0$. Then $\psi_2-\psi u=(1-u)\psi_2-\psi_1 u\geq 0$ and (\ref{keyineq}) allows us to say that
$g_u=g(u)$ and that $\psi_2-\psi u=0$. In particular, we have since $\psi_2-\psi u=(1-u)\psi_2-\psi_1 u$:
\begin{itemize}
\item In the set $\{u=0\}$, then $\psi_2-\psi u=0$ implies that $\psi_2=0$ and so $\psi=\psi_1\leq 0$.
\item In the set $\{u=1\}$, then $\psi_2-\psi u=0$ implies that $\psi_1=0$ and so $\psi=\psi_2\geq 0$.
\item In the set $\{0<u<1\}$, then $\psi_2-\psi u=0$ implies that $\psi_1=\psi_2=0$ and so $\psi=0$.
\end{itemize}
In this manner, $\psi\in \partial I_{[0,1]}(u)$. Going back to (\ref{lastineq}), we have
\begin{align}\label{240607_01}
\limsup_{m\rightarrow+\infty}\int_0^Te^{-ct}\E\left[||\uhnr(t)||^2_{L^2(\Lambda)}\right]\,dt\leq \int_0^T e^{-ct} \E\left[||u(t) ||^2_{L^2(\Lambda)}\right]\,dt,
\end{align}
and one can thus conclude from \eqref{240607_01} and Proposition \ref{addreg u} that $(\uhnr)_m$ converges strongly to $u$ in $L^2(\Omega;L^2(0,T;L^2(\Lambda)))$, that $\beta_u=\beta(u)$ and that $(u, \psi)$ is the unique variational solution of Problem \eqref{equation} in the sense of Definition \ref{solution}. As a consequence of \cite[Corollary 1.2.23, p.25]{HNVW16} with $(S,\mathcal{A},\mu)=(\Omega,\mathcal{F},\mathds{P})$, $(T,\mathcal{B},\nu)=((0,T),\mathcal{B}(0,T),\lambda)$, $X=L^2(\Lambda)$, $p=2$, $L^2(\Omega;L^2(0,T;L^2(\Lambda)))$ is isometrically isomorphic to $L^2(0,T;L^2(\Omega;L^2(\Lambda)))$, hence $(\uhnr)$ also converges strongly towards $u$ in $L^2(0,T;L^2(\Omega;L^2(\Lambda)))$. Combining this last information with the boundedness of $(\uhnr)_m$ in $L^{\infty}(0,T; L^2(\Omega;L^2(\Lambda)))$ (see Lemma \ref{210611_lem01}) allows us to conclude that $(\uhnr)_m$ converges strongly towards $u$ in $L^{p}(0,T; L^2(\Omega;L^2(\Lambda)))$ for any finite $p\geq 1$ thanks to Vitali's theorem \cite[Corollaire 1.3.3]{D}.
\end{proof}

\textbf{Acknowledgments} The authors would like to thank J.~Droniou and T.~Gallou\"et for
their valuable suggestions. This work has been supported by the German Research Foundation project, Germany (ZI 1542/3-1) and various Procope programs: Project-Related Personal Exchange France-Germany, Germany (49368YE), Procope Mobility Program, France (DEU-22-0004 LG1), Procope Plus Project, France (0185-DEU-22-001 LG 4) and the ANR project SOS2ID: Stochastic Optimization Schemes - Infinite Dimensional and Inertial Dynamics, France (ANR-24-CE40-3786).
\bibliographystyle{plain}
\bibliography{Bauzet_Sultan_Vallet_Zimmermann_AllenCahn_250422.bib}
\end{document}